\documentclass[10pt]{article}
\usepackage[utf8]{inputenc}
\usepackage{amsmath,amssymb,amsfonts,amsthm,mathtools,xfrac,extarrows,enumitem,bbm}
\usepackage{epstopdf}
\usepackage{graphicx, float}
\usepackage{amscd}
\usepackage[active]{srcltx}
\usepackage[font={small}]{caption}
\usepackage{epstopdf}
\usepackage{amscd}
\usepackage{hyperref}
\usepackage{color}
\allowdisplaybreaks[4]

\newcommand{\R}{\mathbb{R}}

\newcommand\extrafootertext[1]{%
  \bgroup
  \renewcommand\thefootnote{\fnsymbol{footnote}}%
  \renewcommand\thempfootnote{\fnsymbol{mpfootnote}}%
  \footnotetext[0]{#1}%
  \egroup
}

\newtheorem{thm}{Theorem}
\newtheorem{lemma}{Lemma}[section]
\newtheorem{corollary}[lemma]{Corollary}
\newtheorem{remark}[lemma]{Remark}
\newtheorem{proposition}[lemma]{Proposition}
\newtheorem{definition}[lemma]{Definition}

\numberwithin{equation}{section}

\title{Prescribing scalar curvatures:\\ loss of minimizability}

\author
{
Martin Mayer
\\
\small{Scuola Superiore Meridionale,
Largo San Marcellino 10, 80138 Napoli, Italia
}
\\
\&
\\
Chaona Zhu
\\
\small{School of Mathematics and Statistics, Ningbo University}
\\
\small{No. 818, Fenghua Road, Ningbo 315211, P.R. China}
}

\begin{document}
\maketitle

\begin{abstract}
Prescribing conformally the scalar curvature on a closed manifold with negative Yamabe invariant as a given function $K$
is possible under  smallness assumptions on $K_{+}=\max\{K,0\}$ and in particular, when $K<0$.
In addition, while solutions are unique in case $K\leq 0$, non uniqueness generally holds,
when $K$ is sign changing and $K_{+}$ sufficiently small and flat around its critical points.
These solutions are found variationally as minimizers.
Here we study, what happens, when the relevant arguments fail to apply,
describing on one hand the loss of minimizability generally,
while on the other we construct a function $K$,
for which saddle point solutions to the conformally prescribed scalar curvature problem still exist.
\end{abstract}	

{\footnotesize
\begin{center}
nonlinear analysis
\\
\end{center}
}

\tableofcontents

\section{Introduction}

Following our previous work \cite{Mayer_Zhu_Negative_Yamabe_1},
we continue the investigation of the conformally
prescribed scalar curvature problem
\begin{equation}\label{prescribed_equation_1}
R_{g_{u}}=K \in C^{\infty}(M) \; \text{ for } \; g_{u}=u^{\frac{4}{n-2}}g_{0}
\end{equation}
on a closed Riemannian manifold $M=(M^{n},g_{0})$ of dimension $n\geq 3$
and negative conformal Yamabe invariant
$$
Y(M)
=
\inf_{\substack{u\in H^{1}(M)\\ u>0}}
\frac
{\int_ML_{g_0}uud\mu_{g_0}}
{(\int_Mu^{\frac{2n}{n-2}}d\mu_{g_0})^{\frac{n-2}{n}}}
<0.
$$
As is well known, given functions $0<v,w\in C^{\infty}(M)$ and the conformal metric
\begin{equation}\label{conf_metric}
  g_{w}=w^{\frac{4}{n-2}}g_0
\end{equation}
with induced scalar curvature $R_w=R_{g_{w}}$, the conformal Laplacian
$$
L_{g_{w}}=-c_{n}\Delta_{g_{w}}+R_{g_{w}}, \; c_n=\frac{4(n-1)}{n-2}
$$
satisfies the conformal covariance property
$$
L_{g_{w}}v=w^{-\frac{n+2}{n-2}}L_{g_{0}}(wv),
$$
whence, setting $v=1$ and $w=u$, \eqref{prescribed_equation_1} is equivalent finding a positive solution $u>0$ for the
critical equation
\begin{equation}\label{prescribed_equation_2}
L_{g_{0}}u
=
Ku^{\frac{n+2}{n-2}},\; u>0.
\end{equation}
Equation \eqref{prescribed_equation_2} is always solvable,
if the function $K$ on $M$ to be prescribed is strictly negative
\cite{Kazdan_Warner_JDE}, and so we may assume
\begin{equation}\label{Yamabe_metric}
  R_{g_0}=-1\; \text{ and }\; L_{g_0}=-c_n\Delta_{g_0}-1.
\end{equation}
Moreover, sufficient and necessary conditions are known
\cite{Ouyang, Vazquez_Veron}, if $K\leq 0$.
On the other hand $\min K<0$ is necessary for solvability
\cite{Kazdan_Warner_JDE},
and so we are considering here, as in \cite{Mayer_Zhu_Negative_Yamabe_1}, a sign changing, smooth function $K$.
While it is easy to find a plethora non prescribable sign changing functions,
solvability can still be guaranteed, if $K$ is \textit{not too positive},
as shown in \cite{Aubin_Bismuth,Mayer_Zhu_Negative_Yamabe_1,Rauzy_Existence}.
In that case we actually find \cite{Mayer_Zhu_Negative_Yamabe_1} a solution $u_{0}>0$ to \eqref{prescribed_equation_2} as a minimizer of a naturally associated
variational energy $J$, cf.\eqref{J}, with negative mean scalar curvature
$$
\int R_{g_{u_{0}}}d\mu_{g_{u_{0}}}<0,
$$
where $d\mu_{g_{w}}=w^{\frac{2n}{n-2}}d\mu_{g_{0}}$ denotes the induced measure density, see \eqref{conf_metric}.
Complementarily solutions are in case $K\leq 0$  unique
\cite{Aubin_Sur_Le_Problem, Kazdan_Warner_JDE},
while for $K$ sign changing a second solution may exist,
as confirmed in \cite{ Mayer_Zhu_Negative_Yamabe_1,Rauzy_Multiplicity}
under the assumptions, that
$K$ is \textit{not too positive} and \textit{sufficiently flat} around a global maximum point,
in which case a second solution $0<u_{1}\neq u_{0}$ of \eqref{prescribed_equation_2} can be found
\cite{Mayer_Zhu_Negative_Yamabe_1}
as a minimizer of a second functional $I$, inducing a \textit{positive} mean scalar curvature. The generic case without such flatness assumptions will be discussed in
\cite{Mayer_Zhu_Negative_Yamabe_3}.

Here we study, what happens, when the sufficient condition to guarantee minimizability, namely the validity of some A-B-inequality, cf. \eqref{AB} is lost.
In this case we associate to $J$,
defined on a naturally related, contractible variational space $X$,
an exit set $E\neq \emptyset$, onto which portions of augmented sublevels
$$
\{ J\leq L \} \cup E
$$
by deformation always retract.
However, if we assume $\{\partial J=0\}=\emptyset$, i.e. absence of solutions to \eqref{prescribed_equation_2}, then \textit{all} augmented sublevels, which are contractible, retract by
deformation onto E.
As a consequence non solvability necessitates contractibility of $E$,
and yet we construct a function, whose exit set is neither empty nor
contractible, whence we still derive the existence of a saddle point type solution.


\

\noindent
To be precise, as in \cite{Mayer_Zhu_Negative_Yamabe_1} we consider the scaling invariant functional
\begin{equation}\label{J}
J(u)=\frac{-k_u}{(-r_u)^{\frac{n}{n-2}}}>0
\end{equation}
on the variational space
\begin{equation}\label{definition_of_X}
X=\{u>0\}\cap\{r_u<0\}\cap\{k_u<0\}\cap\{\Vert u\Vert _{L_{g_0}^{\frac{2n}{n-2}}}=1\}
\subset C^{\infty}(M)
,
\end{equation}
where $K\in C^{\infty}(M)$ changes sign,
$k_u=k_{g_{u}}=\int Ku^{\frac{2n}{n-2}}d\mu_{g_0}$
and
\begin{equation}\label{r_and_k_definitions}
r_u=r_{g_{u}}=\int R_{g_u}d\mu_{g_u}=\int L_{g_0}uud\mu_{g_0}=
c_{n}\int \vert \nabla u \vert^{2}d\mu_{g_{0}} - \int u^{2}d\mu_{g_{0}},
\end{equation}
with derivative
\begin{equation}\label{partial_J}
\partial J(u)
=
\frac{2^*}{(-r_u)^{\frac{n}{n-2}}}\bigg(\frac{-k_u}{-r_u}\, L_{g_0}u-Ku^{\frac{n+2}{n-2}}\bigg)
,\;
2^*=\frac{2n}{n-2}
\end{equation}
and a Yamabe type flow
\begin{equation}\label{flow_for_J}
\partial_{t}u
=
-(\frac{-k_u}{-r_u}R_u-K)u
=
-u^{-\frac{4}{n-2}}(\frac{-k_u}{-r_u}L_{g_{0}}u-Ku^{\frac{n+2}{n-2}}).
\end{equation}
Hence a critical point of $J$ corresponds to a solution of \eqref{prescribed_equation_2},
and in \cite{Mayer_Zhu_Negative_Yamabe_1} we prove, that
$J$ achieves a global minimum on $X$, if an A-B inequality holds true at least on some sublevel set of $J$,
while the validity of a global A-B-inequality can be guaranteed under suitable assumptions on $K$.
To be precise, let
$$\nu_{1}(L_{g_{0}},D )=\sup_{D \subset \Omega \; \text{smooth } \; }
\nu_{1}(L_{g_{0}},\Omega)$$
denote the first Dirichlet eigenvalue, where for a smooth subset $\Omega \subset M$
\begin{equation*}
\nu_{1}(L_{g_{0}},\Omega)
=
\inf_\mathcal{A}\frac{\int_{\Omega}L_{g_0}uud\mu_{g_0}}{\int_{\Omega}u^2d\mu_{g_0}}
,
\;
\mathcal{A}=\{u\in C_{0}^\infty(\Omega) \; : \; u>0 \; \text{ in } \; \Omega\ \}.
\end{equation*}

\begin{proposition}[\cite{Mayer_Zhu_Negative_Yamabe_1}]\label{prop_A_B_inequality_from_A_B_conditions}
 There exists $\epsilon>0$ such, that
for any $K\in C^{\infty}(M)$, if
$$
\{ K\geq 0 \} = \Omega_{K} \subset \subset \Omega \subset \subset D
$$
with smooth $\Omega,D\subset M$ and
\begin{enumerate}[label=(\roman*)]
\item $\nu_{1}(D)=\nu_{1}(L_{g_{0}},D)>0$
\item
$
\sup_M K
<
\epsilon
[
\; dist^{2\frac{n-1}{n-2}}(\partial \Omega,\partial D)
(\frac{\nu_{1}(D)}{\nu_{1}(D)+1})^{\frac{n}{n-2}}
\;
]
\inf_{M\setminus \Omega} (-K)
$,
\end{enumerate}
then for some constants $A,B>0$ there holds
\begin{equation}\label{AB}
\Vert u \Vert_{H^{1}}^{2}
\leq
Ar_{u}
+
B\vert k_{u} \vert^{\frac{n-2}{n}}
\end{equation}
for all $u\in H^{1}(M)$.
\end{proposition}
We call \eqref{AB} an A-B-inequality and global, if valid on $H^{1}(M)$, as ensured by Proposition \ref{prop_A_B_inequality_from_A_B_conditions},
if $K$ is \textit{not too positive} in the sense of (i) and (ii) above.
\begin{thm}[\cite{Mayer_Zhu_Negative_Yamabe_1}]\label{thm_minimize_J}
If an A-B-inequality holds
on some sublevel
$\{ J\leq L \} \neq \emptyset$,
then $J$ admits a global minimizer on $X$,
which is a solution of equation \eqref{prescribed_equation_2}.
\end{thm}
\noindent
Note, that, if \eqref{AB} holds, then readily
$$E=\{ k=0 \}\cap \{ r<0 \} \cap \{u>0\}\cap \{\|u\|_{L_{g_0}^{\frac{2n}{n-2}}}=1\}= \emptyset.$$
Corollary \ref{cor_exit_set_versus_A_B_inequality} actually tells us, that
an A-B-inequality holding on some sublevel
is equivalent to $E=\emptyset$.
Conversely, as we discuss in Section \ref{Section_Contractibility_of_the_Exit_Set},
if we assume $\{\partial J=0\}=\emptyset$, which for many functions $K$ is the case,
then $X\cup E$ retracts naturally along an energy decreasing flow onto $E$,
whence we refer to $E$ as the exit set, and $E$ inherits the contractibility of $X$.
Also note, that if $\{\partial J=0\}=\emptyset$, then the classical \textit{pure}, i.e. zero weak limit,
blow-up via bubbling does not occur and so the theory of critical point at infinity
\cite{Bahri_Critical_Points_At_Infininty_In_The_Variational_Calculus, MM3}
does not apply, cf. \cite{Mayer_Zhu_Negative_Yamabe_1}.

Thus we are led to ask, which  conditions on a function $K$
still guarantee $\{\partial J=0\}\neq\emptyset$ in absence of an A-B-inequality, i.e. when $E\neq \emptyset$.
To shed some light on this question, 
in Section \ref{sec_non_connectedness} 
we construct a double peak type function $K=K_{dp}$,
precisely defined in \eqref{K_in_the_exit_set_example},
for which $E\neq \emptyset$ and Theorem \ref{thm_minimize_J} is therefore not applicable,
while $E$ is not contractible.
Hence by topological obstruction $\{ \partial J=0 \}=\emptyset $ is impossible
and a solution to the conformally prescribed scalar curvature problem \eqref{prescribed_equation_1} still exists.

\

\textbf{Hypotheses} \;
Throughout this work we will assume, that
\begin{enumerate}[label=(\roman*)]
\item[(H1)] the background metric $g_0$ is smooth with $R_{g_0}=-1$
\item[(H2)] the first Dirichlet eigenvalue is positive,
$
\nu_{1}=\nu_{1}(L_{g_{0}}, \{K\geq 0\} )>0
$

\item[(H3)] the conformal Laplacian is invertible,
$ker(L_{g_{0}})=\{0\}$.
\end{enumerate}

As is well known, we may satisfy (H1) on every closed manifold;
and (H2) is a necessary condition for solvability of \eqref{prescribed_equation_2}, cf. \cite{Rauzy_Existence}.
Moreover, while (H3) is a generic property, cf. \cite{Mayer_Zhu_Negative_Yamabe_1}, it
is probably just a technical assumption.

\smallskip

Assuming (H1)-(H3), for a smooth sign changing function $K$ our main theorem states as follows.

\begin{thm}\label{thm_exit_set_example}
For $J=J_{K}$ and the exit set $E$ of $J$ there holds
\begin{enumerate}[label=(\roman*)]
\item $E=\emptyset$, if and only if $J$ attains its infimum, in particular $\{ \partial J = 0\} \neq \emptyset$.
\item If $\{ \partial J=0 \}=\emptyset$, then $X\cup E$ is contractible and retracts onto $E$ by strong deformation. In particular $E$ is contractible.
\item For the double peak type function $K=K_{dp}$
the exit set $E\neq \emptyset$ contains at least two distinct connected components,
in particular $\{ \partial J =0 \} \neq \emptyset$.
\end{enumerate}
\end{thm}

While $(i)$ and $(ii)$ of Theorem \ref{thm_exit_set_example}, proved in Section \ref{Section_Contractibility_of_the_Exit_Set}, \textit{geometrically} describe the solvability of $\partial J=0$, part $(iii)$ or equivalently Proposition \ref{prop_disconnectedness_of_E_for_K_dp} show, that even when minimizability of $J$ is not feasible, we may still hope to tell from the \textit{shape} of $K$, that a solution exists.

\begin{remark}
\begin{enumerate}[label=(\roman*)]
\item 
In terms of
(variational) space, flow and exit set
the ideas and arguments, leading to Theorem 2, 
are reminiscent to those in \cite{Conley_Book}.
\item We are not aware of previous existence results in the context of
conformal geometry, based on a \emph{computation} of a non trivial exit set,
and thus expect the arguments, developed here, to apply to related problems.
\end{enumerate}
\end{remark}

To put Theorem \ref{thm_exit_set_example} into context, we recall from \cite{Kazdan_Warner_JDE,Mayer_Zhu_Negative_Yamabe_1}, that positivity of the unique solution
to the linear equation
\begin{equation*}
\mathcal{L}_{g_{0}}\bar{w}=-(n-1)\Delta_{g_{0}}\bar{w}+\bar{w}=-K
\end{equation*}
is a necessary condition for solvability of \eqref{prescribed_equation_2}, cf. \eqref{Yamabe_metric}.
Considering thus a sign changing function $w\in C^{\infty}(M)$ and letting
$$
K_{1}=-\mathcal{L}_{g_{0}}w
\; \text{ and  } \;
0>K_{0} \in C^{\infty}(M),
$$
we then can solve \eqref{prescribed_equation_2} for $K=K_{0}$, but not for $K=K_{1}$. Interpolating via
$$
K_{\tau}=(1-\tau)K_{0}+\tau K_{1},\; \tau \in [0,1]
$$
we then find
\begin{enumerate}[label=(\roman*)]
\item 	solvability for $\tau\geq 0$ close to $0$ via minimizing $J_{\tau}=J_{K_{\tau}}$ and $E_{\tau}=\emptyset$
\item non solvability for $\tau \leq 1$ close to $1$ and,
provided (H2) holds for $K_{1}$,
$$
X_{\tau}\cup E_{\tau}=\{ k_{\tau}\leq 0 \}\cap \{ r_{\tau}<0\} \xhookrightarrow{sdr}
E_{\tau}=\{ k_{\tau}=0\} \cap \{ r_{\tau}<0 \} \neq \emptyset
$$
as a strong deformation retract with  $E_{\tau}$ contractible.
\end{enumerate}
Theorem \ref{thm_exit_set_example} now tells us,
that the loss of solvability is not just related
to the existence, but also to the topology of the related exit sets,
whose programmatic study remains elusive.

\section{Preliminaries}

Here we recall some previous results, standard tools and notations.

\smallskip

\noindent {\bf Notation.}
We denote by $O(1)$ and $O^+(1)$ any quantity and strictly positive quantity respectively, which are bounded, and for $b\geq 0$ define
$O(b)=b\cdot O(1)$, $O^+(b)=b\cdot O^+(1)$. Similarly
$o_{a}(1)$ and $o^+_{a}(1)$ for $a>0$ denotes any quantity and any strictly positive quantity respectively, which tend to zero, as $a\longrightarrow 0$, while
$o_{a}(b)=b\cdot o_{a}(1)$ and $o^+_{a}(b)=b\cdot o^+_{a}(1)$ for $b\geq 0$.
For brevity we say
$a=b$ up to $O(d)$ or $o_{c}(d)$,
if $a=b+O(d)$ or $a=b+o_{c}(d)$  respectively.
Finally let
$$\|\cdot\|=\|\cdot\|_{W^{1,2}(M, g_0)} \;\text{ and }\; \|\cdot\|_{L^p}=\|\cdot\|_{L_{g_0}^p}
$$
and observe, that due to $X \subset \{ r<0 \} \cap \{ \Vert u \Vert_{L^{\frac{2n}{n-2}}}=1 \}$, cf. \eqref{r_and_k_definitions},  we have
\begin{equation}
\label{equ_norm}
\Vert \cdot \Vert \simeq \Vert \cdot \Vert_{L^{2}} \simeq \Vert \cdot \Vert_{L^{\frac{2n}{n-2}}}=1 \; \text{ on } \; X.
\end{equation}
With these notations at hand the properties of {\em conformal normal coordinates}, cf. \cite{Guenther_Conformal_Normal_Coordinates,Lee_Parker_Yamabe_Problem},
read as follows.
Given $a \in M$, we may a choose conformal metric
\begin{equation*}
\begin{split}
g_{a}=u_{a}^{\frac{4}{n-2}}g_{0}
\; \text{ with } \;
u_{a}=1+O(d^{2}_{g_{0}}(a,\cdot)),
\end{split}
\end{equation*}
whose volume element in $ g_{a}$-geodesic normal coordinates coincides with the Euclidean one \cite{Guenther_Conformal_Normal_Coordinates}. In particular
$R_{g_{a}}=O(d^{2}_{g_{0}}(a,\cdot))$ for the scalar curvature and
\begin{equation*}
\begin{split}
(\exp_{a}^{g_{0}})^{-}\circ \exp_{a}^{g_{a}}(x)
=
x+O(\vert x \vert^{3})
\end{split}
\end{equation*}
for the exponential maps centered at $ a $.
We then denote by $r_a$ the geodesic distance from $a$ with respect to the metric $g_a$
just introduced.
With these choices  the expression of the
Green's function $G_{g_{ a }}$  for the conformal
Laplacian $L_{g_{a}}$ with pole at $a \in M$, denoted by
$G_{a}=G_{g_{a}}(a,\cdot)$, simplifies to
\begin{equation*}
\begin{split}
G_{ a }=\frac{1}{4n(n-1)\omega _{n}}(r^{2-n}_{a}+H_{ a }), \; r_{a}=d_{g_{a}}(a, \cdot)
, \;
H_{ a }=H_{r,a }+H_{s, a },
\end{split}
\end{equation*}
where $\omega_{n} = |S^{n-1}|$, cf. Section 6 in \cite{Lee_Parker_Yamabe_Problem}.
Here $H_{r,a }\in C^{2, \alpha}$,
while
\begin{equation}\label{Irregular_Part}
\begin{split}
H_{s,a}
=
O
\begin{pmatrix}
0 & \text{ for }\, n=3
\\
r_{a}^{2}\ln r_{a} & \text{ for }\, n=4
\\
r_{a}& \text{ for }\, n=5
\\
\ln r_{a} & \text{ for }\, n=6
\\
r_{a}^{6-n} & \text{ for }\, n\geq 7
\end{pmatrix}
\end{split}
\end{equation}
and
$H_{s,a}\equiv 0$, if $g_{a}$ is flat around $a$, cf. \cite{Mayer_Zhu_Negative_Yamabe_1}.
On $\{ G_{a}>0 \} $ let
for $\lambda>0$
\begin{equation*}
\theta_{a,\lambda}=u_{ a }(\frac{\lambda}{1+\lambda^{2}\gamma_{n} G_{a}^{\frac{2}{2-n}}})^{\frac{n-2}{2}}
,\;
\gamma_{n}=(4n(n-1)\omega _{n})^{\frac{2}{2-n}},
\end{equation*}
cf. \cite{MM1} or \cite{Mayer_Scalar_Curvature_Flow}. Extend $\theta_{a,\lambda}=0$ on $\{ G_{a}\leq 0 \} $  and
with a smooth cut-off function
\begin{equation}\label{cut_off_function}
\eta_{a}=\eta(d_{g_{a}}(a,\cdot))
=
\left\{
\begin{matrix*}
1  &  \;\text{on}\;  &  B_{\epsilon}(a) &=&B^{d_{g_{a}}}_{\epsilon}(a) \\
0 &  \;\text{on}\;  & B_{2\epsilon}(a)^{c}&=&M\setminus  B^{d_{g_{a}}}_{2\epsilon}(a) , \\
\end{matrix*}
\right.
\end{equation}
where $0<\epsilon\ll 1$ is independent of $a\in M$ and such, that on
$B^{d_{g_{a}}}_{2\epsilon}(a)$ the conformal normal coordinates
from $g_{a}$ are well defined and $G_{g_{a}}>0$, define
\begin{equation}\label{Bubble_Definition}
\begin{split}
\varphi_{a, \lambda }
= &
\eta_{a}\theta_{a,\lambda} \geq 0.
\end{split}
\end{equation}
Note,  that
$
\gamma_{n}G^{\frac{2}{2-n}}_{ a }(x) =
d_{g_a}^{2}(a,x) + o(d_{g_a}^{2}(a,x))$,
as $x \longrightarrow a$.
\begin{lemma}[\cite{Mayer_Zhu_Negative_Yamabe_1}]\label{lem_L_g_0_of_bubble}
We have,
\begin{equation*}
L_{g_{0}}\varphi_{a, \lambda}
=
4n(n-1)\varphi_{a, \lambda}^{\frac{n+2}{n-2}}
+
O(\frac{\chi_{B_{2\epsilon}(a)\setminus B_{\epsilon}(a)}}{\lambda^{\frac{n-2}{2}}})
+
o_{\frac{1}{\lambda}}(\frac{1}{\lambda^2}+\frac1{\lambda^{\frac{n-2}{2}}}) \; \text{ in } \; W^{-1,2}(M).
\end{equation*}
The expansion above persists upon taking the $\lambda \partial_\lambda$ and $\frac{\nabla_{a}}{\lambda}$ derivatives.
\end{lemma}
\begin{proof}
As in Lemma 3.3 in \cite{Mayer_Scalar_Curvature_Flow} or Lemma 4.2 in \cite{Mayer_Zhu_Negative_Yamabe_1} we find
\begin{equation*}
\begin{split}
& L_{g_{0}}  \varphi_{a,\lambda}
= \;
4n(n-1)\varphi_{a,\lambda}^{\frac{n+2}{n-2}}\\
&
-c_{n}\Delta_{g_{0}}\eta_{a}\theta_{a,\lambda}-2c_{n}\langle\nabla \eta_{a},\nabla \theta_{a,\lambda}\rangle_{g_{0}}
+4n(n-1)(\eta_{a}-\eta_{a}^{\frac{n+2}{n-2}})\theta_{a,\lambda}^{\frac{n+2}{n-2}}
\\
&
-
2nc_{n}
(1+o_{r_{a}}(1))r_{a}^{n-2}((n-1)H_{a}+r_{a}\partial_{r_{a}}H_{a}) \eta_a \theta_{a, \lambda}^{\frac{n+2}{n-2}} +
\frac{u_{a}^{\frac{2}{n-2}}R_{g_{a}}}{\lambda}\eta_a\theta_{a, \lambda}^{\frac{n}{n-2}},
\end{split}
\end{equation*}
where the terms of the second line above can be pointwise subsumed under some
$$O(\lambda^{\frac{2-n}{2}})\chi_{B_{2\epsilon}(a)\setminus B_{\epsilon}(a)}.$$
Moreover,
recalling $R_{g_{a}}=O(r_{a}^{2})$ and \eqref{Irregular_Part},
those of the third line are of order
$$\left\{
\begin{matrix*}[l]
o_{\frac1\lambda}(\lambda^{\frac{2-n}{2}})  &  \;\text{ for }\;  &  3\leq &n& \leq 5 &\\
o_{\frac1\lambda}(\lambda^{-2}) &  \;\text{ for }\;  &  &n&\geq 6 &
\end{matrix*}
\right.
$$
in $W^{-1,2}$.
\end{proof}

\begin{remark}
 Note, that in contrast to our previous paper \cite{Mayer_Zhu_Negative_Yamabe_1}, where we assume the flatness condition $Cond_{n}$, here we have an additional error term
 $$o_{\frac{1}{\lambda}}(\lambda^{-2}),$$
 which is of no concern, as we now target $\lambda^{-2}+\lambda^{-\frac{n-2}{2}}$ as the level of precision.
\end{remark}

\noindent {\bf Notation.}
For $k,l=1,2,3$ and $ \lambda_{i} >0, \, a _{i}\in M, \,i= 1, \ldots,q$ let
\begin{enumerate}[label=(\roman*)]
 \item
$
\varphi_{i}
=
\varphi_{a_{i}, \lambda_{i}}$ and $(d_{1,i},d_{2,i},d_{3,i})
=
(1,-\lambda_{i}\partial_{\lambda_{i}}, \frac{1}{\lambda_{i}}\nabla_{a_{i}})
$
 \item
$\phi_{1,i}=\varphi_{i}, \;\phi_{2,i}
=
-\lambda_{i} \partial_{\lambda_{i}}\varphi_{i},
\;\phi_{3,i}= \frac{1}{\lambda_{i}} \nabla_{ a _{i}}\varphi_{i}$, so
$
\phi_{k,i}=d_{k,i}\varphi_{i}.
$
\end{enumerate}
Note, that with the above definitions $\phi_{k,i}$ is uniformly bounded in $H^{1}(M)$. Moreover we have the following standard inter- and selfaction
estimates;  we refer to Lemma 4.3 in \cite{Mayer_Zhu_Negative_Yamabe_1} and the details in Lemma 3.5\footnote
{
See also Lemma 3.5 and its proof in the more
detailed version of \cite{Mayer_Scalar_Curvature_Flow}
found at
\url{http://geb.uni-giessen.de/geb/volltexte/2015/11691/}
}
in \cite{Mayer_Scalar_Curvature_Flow}.

\begin{lemma}[\cite{Mayer_Zhu_Negative_Yamabe_1}]\label{lem_interactions}
Let $k,l=1,2,3$ and $i,j = 1, \ldots,q$. Then for
\begin{equation*}
\varepsilon_{i,j}
=
\eta(d_{g_{0}}(a_{i},a_{j}))
(
\frac{\lambda_{j}}{\lambda_{i}}
+
\frac{\lambda_{i}}{\lambda_{j}}
+
\lambda_{i}\lambda_{j}\gamma_{n}G_{g_{0}}^{\frac{2}{2-n}}(a _{i},a _{j})
)^{\frac{2-n}{2}}
\end{equation*}
with a suitable cut-off function
$$
\eta
=
\left\{
\begin{matrix*}[l]
1  &  \;\text{on}\;  &  r<4\epsilon \\
0 &  \;\text{on}\;  &  r\geq 6 \epsilon \\
\end{matrix*}
\right.
$$
and $\epsilon>0$ sufficiently small there holds
\begin{enumerate}[label=(\roman*)]
\item
$
\vert \phi_{k,i}\vert,
\vert \lambda_{i}\partial_{\lambda_{i}}\phi_{k,i}\vert,
\vert \frac{1}{\lambda_{i}}\nabla_{a_{i}} \phi_{k,i}\vert
\leq
C \chi_{\{\eta_{a_i}>0\}}\theta_{i} $, cf. \eqref{cut_off_function}

\item
$
\int \varphi_{i}^{\frac{4}{n-2}}
\phi_{k,i}\phi_{k,i}d\mu_{g_{0}}
=
c_{k}\cdot id
+
O(\frac{1}{\lambda_{i}^{2}}+\frac{1}{\lambda_{i}^{n-2}})
, \;c_{k}>0$

\item
for  $i\neq j$ up to some $o_\epsilon(
\varepsilon_{i,j}+\frac{1}{\lambda_{i}^{2}})$
\begin{equation*}
\int \varphi_{i}^{\frac{n+2}{n-2}}\phi_{k,j}d\mu_{g_{0}}
=
b_{k}d_{k,i}\varepsilon_{i,j}
=
\int \varphi_{i}d_{k,j}\varphi_{j}^{\frac{n+2}{n-2}}  d\mu_{g_{0}}
\end{equation*}
 \item
$
\int \varphi_{i}^{\frac{4}{n-2}}
\phi_{k,i}\phi_{l,i}d\mu_{g_{0}}
=
O(\frac{1}{\lambda_{i}^{2}}+\frac{1}{\lambda_{i}^{{n-2}}})$
for $k\neq l$ and for $k=2,3$
$$
\int \varphi_{i}^{\frac{n+2}{n-2}}
\phi_{k,i}d\mu_{g_{0}}
=
O(\frac{1}{\lambda_{i}^{{n-2}}})
$$
 \item
$
\int \varphi_{i}^{\alpha}\varphi_{j}^{\beta} d\mu_{g_{0}}
=
O(\varepsilon_{i,j}^{\beta})
$
for $i\neq j,\;\alpha +\beta=\frac{2n}{n-2}, \; \alpha>\frac{n}{n-2}>\beta\geq 1$
\item
$
\int \varphi_{i}^{\frac{n}{n-2}}\varphi_{j}^{\frac{n}{n-2}} d\mu_{g_{0}}
=
O(\varepsilon^{\frac{n}{n-2}}_{i,j}\ln \varepsilon_{i,j}), \,i\neq j
$
 \item
$
(1, \lambda_{i}\partial_{\lambda_{i}}, \frac{1}{\lambda_{i}}\nabla_{a_{i}})\varepsilon_{i,j}=O(\varepsilon_{i,j})
+
o_{\sfrac1{\lambda_i}+\sfrac1{\lambda_j}}(\frac{1}{\lambda_{i}^{\frac{n-2}{2}}}+\frac{1}{\lambda_{j}^{\frac{n-2}{2}}})
, \,i\neq j$,
\end{enumerate}
with constants $ b_{1}=b_{2}=b_{3}=\underset{\R^{n}}{\int}\frac{dx}{(1+r^{2})^{\frac{n+2}{2}}}=b_0$ and
$$
c_{1}=\underset{\R^{n}}{\int}\frac{dx}{(1+r^{2})^{n}},\,
c_{2}=\frac{(n-2)^{2}}{4}\underset{\R^{n}}{\int}\frac{( r^{2}-1)^{2}dx}{(1+r^{2})^{n+2}},\,
c_{3}={(n-2)^{2}}\underset{\R^{n}}{\int}\frac{r^{2}dx}{(1+r^{2})^{n+1}}.
$$
\end{lemma}

\

\begin{definition}\label{V_p_e}
For $\varepsilon>0$ and $ u\in H^{1}(M)$ let
\begin{flalign}
&
(i)
&
A_{u}(p, \varepsilon)
 =
\{ \;
(
\alpha ,\alpha_i, & \,\lambda_i, a_i)\in \R_+  \times \R_+^p\times \R_{+}^p \times M^p
\; : \;
 \notag \\
& & &  \textstyle\sum_{i} \alpha_i^2\geq\varepsilon^2,\lambda_i^{-1}\leq \varepsilon,\;
\Vert u-\alpha -\alpha_i\varphi_{a_i, \lambda_i}\Vert
\leq\varepsilon
\; \} &
\notag
\end{flalign}	
\vspace{-12pt}
\begin{flalign}
&
(ii) &
U(p, \varepsilon)
=
\{
u\in W^{1,2}(M)
\mid
A_{u}(\varepsilon)\neq \emptyset
\}
\cap
\{\|u\|_{L^{\frac{2n}{n-2}}}=1\}\cap\{u>0\}.
& &\notag
\end{flalign}	
\end{definition}

As in \cite{Mayer_Zhu_Negative_Yamabe_1}, we choose a convenient representation on $U(p, \varepsilon)$.

\begin{lemma}\label{lem_optimal_choice1} 
For every $\varepsilon_{0}>0$  there exists
$$0<\varepsilon_{2}<\varepsilon_{1}<\varepsilon_{0}$$
such, that  for any
$u\in U(p, \varepsilon_2)$
there exists a unique
$$(\alpha, \alpha_i, a_i, \lambda_i)\in U(p, \varepsilon_1),$$
for which, letting
$v=u-\alpha -\alpha^{i}\varphi_{a_{i},\lambda_{i}},
$
\begin{enumerate}[label=(\roman*)]
\item \quad $v \perp_{L_{g_{0}}}span\{1,\varphi_{ a_i,\lambda_i}, i=1,...,p\}$
\item \quad the quantities
\begin{itemize}
\item[($\lambda$)] \quad
$\langle\lambda_{i}\partial_{\lambda_{i}}\varphi_{a_{i},\lambda_{i}},v\rangle_{L_{g_{0}}}$,
$\int u^{\frac{4}{n-2}}\lambda_{i}\partial_{\lambda_{i}}\varphi_{a_{i},\lambda_{i}}vd\mu_{g_{0}}$

\item[($a$)] \quad
$\langle \frac{\nabla_{a_{i}}}{\lambda_{i}}\varphi_{a_{i},\lambda_{i}},v\rangle_{L_{g_{0}}}$,
$\int u^{\frac{4}{n-2}}\frac{\nabla_{a_{i}}}{\lambda_{i}}\varphi_{a_{i},\lambda_{i}}vd\mu_{g_{0}}$
\end{itemize}
are of order
$O
(
\sum_{i}
\frac{1}{\lambda_{i}^{{n-2}}}+ \sum_{i\neq j}\varepsilon_{i,j}^{2}   +  \Vert v \Vert^{2}
)
+\sum_{i}o_{\frac{1}{\lambda_i}}(\frac{1}{\lambda_i^4}).
$
\end{enumerate}
\end{lemma}
\begin{proof}
Following the proof of Lemma 4.6 in \cite{Mayer_Zhu_Negative_Yamabe_1} line by line for $\omega=1$, we obtain
\begin{flalign}
(1) \quad
& \lambda_{i}\partial_{\lambda_{i}}\alpha
=
o_{\varepsilon_2}(\textstyle\sum_{i,j}\vert \lambda_{i}\partial_{\lambda_{i}}\alpha_{j} \vert )
+
O(|\langle \lambda_{i}\partial_{\lambda_{i}}\varphi_{a_{i},\lambda_{i}},1
\rangle_{L_{g_{0}}}|) &\notag
\end{flalign}
\vspace{-18pt}
\begin{flalign}
(2) \quad
& \lambda_{i}\partial_{\lambda_{i}}\alpha_{j}
=
o_{\varepsilon_2}
(\vert \lambda_{i}\partial_{\lambda_{i}}\alpha  \vert
+
\textstyle\sum_{j\neq k=1}^{q}\vert \lambda_{i}\partial_{\lambda_{i}}\alpha_{k} \vert
)
+
O(\varepsilon_{i,j})
,\; i\neq j &\notag
\end{flalign}
\vspace{-18pt}
\begin{flalign}
(3) \quad
& \lambda_{i}\partial_{\lambda_{i}}\alpha_{i}
= \;
o_{\varepsilon_2}(\vert \lambda_{i}\partial_{\lambda_{i}}\alpha \vert
+
\textstyle\sum_{i\neq j =1}^{q}\vert \lambda_{i}\partial_{\lambda_{i}}\alpha_{j} \vert ) &\notag\\
& \quad\quad\quad\quad~~+
\langle v,\lambda_{i}\partial_{\lambda_{i}}\varphi_{a_{i},\lambda_{i}}
\rangle_{L_{g_{0}}}
-
\alpha_{i}\langle \varphi_{a_{i},\lambda_{i}},\lambda_{i}\partial_{\lambda_{i}}\varphi_{a_{i},\lambda_{i}}\rangle_{L_{g_{0}}}. &\notag
\end{flalign}
 Applying Lemmata \ref{lem_L_g_0_of_bubble} and \ref{lem_interactions}, we have
\begin{equation*}
\begin{split}
\sum_{i}
(
\vert \lambda_{i}\partial_{\lambda_{i}}\alpha \vert
& +
\sum_{j}\vert \lambda_{i}\partial_{\lambda_{i}}\alpha_{j} \vert
)\\
= \; &
O(\sum_i|\langle \lambda_{i}\partial_{\lambda_{i}}\varphi_{a_{i},\lambda_{i}},1
\rangle_{L_{g_{0}}}|) +
O(\sum_{i}
\vert
\langle \varphi_{a_{i},\lambda_{i}},\lambda_{i}\partial_{\lambda_{i}}\varphi_{a_{i},\lambda_{i}}\rangle_{L_{g_{0}}} \vert
) \\
&+ O(\sum_{i\neq j}\varepsilon_{i,j}) + O(\sum_i|\langle v,\lambda_{i}\partial_{\lambda_{i}}\varphi_{a_{i},\lambda_{i}}
\rangle_{L_{g_{0}}}|)\\
=\;& O(\sum_{i}\frac{1}{\lambda_{i}^{\frac{n-2}{2}}}+\sum_{i\neq j}\varepsilon_{i,j} + \Vert v \Vert) +\sum_io_{\frac1{\lambda_i}}(\frac{1}{\lambda_i^2}).
\end{split}
\end{equation*}
Arguing as for (39) in the proof of Lemma 4.6 in \cite{Mayer_Zhu_Negative_Yamabe_1}, there holds
\begin{equation*}
  \begin{split}
     \int u^{\frac{4}{n-2}}  v \lambda_{i}\partial_{\lambda_{i}}\varphi_{a_{i},\lambda_{i}}d\mu_{g_{0}}
= \; & O(\sum_{i}
(
\vert \lambda_{i}\partial_{\lambda_{i}}\alpha \vert
+
\sum_{j}\vert \lambda_{i}\partial_{\lambda_{i}}\alpha_{j} \vert
)\|v\|)\\
=\;& O(\sum_{i}\frac{1}{\lambda_{i}^{{n-2}}}+\sum_{i\neq j}\varepsilon^2_{i,j} + \Vert v \Vert^2) +\sum_io_{\frac1{\lambda_i}}(\frac{1}{\lambda_i^4})
  \end{split}
\end{equation*}
and, using Lemma \ref{lem_L_g_0_of_bubble},
\begin{equation*}
\begin{split}
\langle \lambda_{i}\partial_{\lambda_{i}}\varphi_{a_{i},\lambda_{i}},v\rangle_{L_{g_{0}}}
= \; &
4n(n-1)
\int \varphi_{a_{i},\lambda_{i}}^{\frac{4}{n-2}} \lambda_{i}\partial_{\lambda_{i}}\varphi_{a_{i},\lambda_{i}} v d\mu_{g_{0}} \\
& +
\int
(
\lambda_{i}\partial_{\lambda_{i}}L_{g_{0}}\varphi_{a_{i},\lambda_{i}}
-
4n(n-1)\lambda_{i}\partial_{\lambda_{i}}\varphi_{a_{i},\lambda_{i}}^{\frac{n+2}{n-2}}
)
v  d\mu_{g_{0}}\\
=\;& O(\sum_{i}\frac{1}{\lambda_{i}^{{n-2}}}+\sum_{i\neq j}\varepsilon^2_{i,j} + \Vert v \Vert^2) +\sum_io_{\frac1{\lambda_i}}(\frac{1}{\lambda_i^4}).
\end{split}
\end{equation*}
Hence the desired $\lambda$-term estimates. The $a$-terms are treated analogously.
\end{proof}

\section{Contractibility of the Exit Set}
\label{Section_Contractibility_of_the_Exit_Set}

Generally each flow line $u$ on $X$, generated by \eqref{flow_for_J},
may either converge to a solution of $\partial J=0$ or
blow up with a weak limit $0<u_{\infty} \in \R\cdot X$,
which again amounts to a solution of $\partial J=0$,
or tend to leave $X$ through $\{ k=0 \} $.

\

Indeed, recalling Section 2 in \cite{Mayer_Zhu_Negative_Yamabe_1}, the flow, generated by \eqref{flow_for_J},
decreases
$$J=\frac{-k}{(-r)^{\frac{n}{n-2}}},$$

whence $-r_{u}\longrightarrow 0$ necessitates $-k_{u}\longrightarrow 0$ for each flow line. Hence
$$-k_{u}\neq o(1)  \Longrightarrow -r_{u}\neq o(1),$$

while $\Vert u \Vert,-r_{u},-k_{u}$ are uniformly bounded from above on $X$.
Thus, unless
$$-k_{u_{t_{k}}}\longrightarrow 0
$$

for an increasing sequence in time, we may assume uniform bounds
$$0<c<-k_{u},-r_{u}<C<\infty,$$

whence, still according to Section 2 in \cite{Mayer_Zhu_Negative_Yamabe_1}, the corresponding flow line $u$
$$
\begin{matrix*}[l]
\quad  & (1)& \;\text{exists smoothly for all times in $X$}\;  \\
\quad & (2) & \;\text{decreases $J$, while $\inf_{ [0,\infty)}J(u)>0$}\;     \\
\quad & (3) & \; \text{remains uniformly positive, i.e. $\inf_{[0,\infty)\times M}u>0$}
\end{matrix*}
$$

and thus leads to a Palais-Smale sequences. Then Proposition 3.1 in \cite{Mayer_Zhu_Negative_Yamabe_1}

applies with a weak limit $u_{\infty}>0$,
corresponding to a solution to $\partial J=0$.

\

\noindent
So, to show solvability of $\partial J=0$,
we may argue by contradiction, assuming
\begin{equation*}
\; \forall \; u=u(\cdot,u_{0})
\; \text{ solving } \; \eqref{flow_for_J} \; \text{ and } \;
\gamma >0 \;
\; \exists \; t>0
\; : \;
0<-k_{u_{t}}<\gamma.
\end{equation*}
Then, once $0<-k_{u_{t}}\ll 1$ is sufficiently small, we shall change to a different energy decreasing flow,
which in finite time, say $t_{1}>0$, achieves $k_{u_{t_{1}}}=0$. This will naturally lead to the description of an associated exit set $E$ from $X$.

\begin{lemma}\label{transverse}
For every $L>0$ there exist $\gamma_{0},\delta_{0}>0$ such, that for every flow line $u$ on $\{ J \leq L \} $, generated by $\eqref{flow_for_J}$, and
$0<\gamma\leq \gamma_{0}$ we have
$$
u\in \{ -k=\gamma \}  \Longrightarrow \partial_{t}k_{u}>  \delta_{0}.
$$
In particular, if a flow line $u $ hits $\{ -k=\gamma \} $, then transversally.
\end{lemma}
\begin{proof}
We have along \eqref{flow_for_J}
\begin{equation}\label{time_evolution_of_k}
\begin{split}
\partial_t k_{u}
= \; &
\partial_t\int Ku^{\frac{2n}{n-2}} d\mu_{g_{0}}
=
\frac{2n}{n-2}\int Ku^{\frac{n+2}{n-2}}\partial_{t}u  d\mu_{g_{0}}\\
= \; &
-\frac{2n}{n-2}\int K(\frac{-k_{u}}{-r_{u}}R_{u}-K)u^{\frac{2n}{n-2}}d\mu_{g_{0}} \\
= \; &
\frac{2n}{n-2}(\int K^2u^{\frac{2n}{n-2}}d\mu_{g_{0}} - \frac{-k_{u}}{-r_{u}}\int KL_{g_0}uud\mu_{g_{0}}).
\end{split}
\end{equation}
Since we decrease $J$ along \eqref{flow_for_J}, there holds for $u\in \{ -k=\gamma \} \cap \{ J\leq L \} $
$$
\frac{-k_{u}}{-r_{u}}
=
J^{\frac{n-2}{n}}(u)(-k_{u})^{\frac{2}{n}}
\leq
L^{\frac{n-2}{n}} \gamma_0^{\frac{2}{n}},
$$
whence due to \eqref{equ_norm} we find
\begin{equation}\label{transversality_estimate_on_the_r_k_term}
\frac{-k_{u}}{-r_{u}}\int KL_{g_0}uu d\mu_{g_{0}}
\leq
CL^{\frac{n-2}{n}}\gamma_0^{\frac2n}.
\end{equation}
On the other hand we have
\begin{equation}\label{K_sqared_control}
\; \exists \;  \varepsilon_0>0 \; \forall \;  u\in X
\; : \; \int K^2u^{\frac{2n}{n-2}}d\mu_{g_{0}}
\geq
\varepsilon_0.
\end{equation}
Indeed and otherwise there exists a sequence
$(u_i)\subset X$ with
$$\int K^2u_i^{\frac{2n}{n-2}}d\mu_{g_{0}}\xlongrightarrow{i\to\infty}  0$$
and again due to
\eqref{equ_norm} we may assume
\begin{equation}\label{weaklim}
  u_i \xrightharpoondown{i\to \infty} u_\infty
\; \text{ weakly in } \;
H^{1}
\; \text{ and } \;
u_i\xlongrightarrow{i\to\infty} u_{\infty}\neq 0
\; \text{ in } \; L^{2}.
\end{equation}
By weak lower semicontinuity we then find
$$
\int K^2u_\infty^{\frac{2n}{n-2}}d\mu_{g_{0}}
\leq
\liminf_{i\to\infty}\int K^2u_i^{\frac{2n}{n-2}}d\mu_{g_{0}}
=
0,
$$
whence $supp(u_\infty) \subseteq \{K=0\}\subset\{K\geq0\}$.
Thus from $\nu_{1}=\nu_1(\{ K\geq 0 \} )>0$
\begin{equation*}
\begin{split}
\nu_{1}\Vert u_{\infty} \Vert_{L^{2}}^{2}
\leq \; &
\int_{\Omega_K}L_{g_0}u_\infty u_\infty d\mu_{g_{0}}
=
\int  L_{g_0}u_\infty u_\infty d\mu_{g_{0}} \\
\leq \; &
\liminf_{i\to\infty}\int  L_{g_0}u_iu_i d\mu_{g_{0}}
=
\liminf_{i\to\infty}r_{u_{i}}
\leq
0,
\end{split}
\end{equation*}
so $u_{\infty}=0$, in contradiction to \eqref{weaklim} and \eqref{K_sqared_control} is proved.
Inserting \eqref{transversality_estimate_on_the_r_k_term}
for sufficiently small $\gamma_{0}>0$
and  \eqref{K_sqared_control}  into
\eqref{time_evolution_of_k}, the assertion readily follows.
\end{proof}

Lemma \ref{transverse} will allow us to combine \eqref{flow_for_J} with moving along
\begin{equation}\label{flow_2}
\partial_t u  = (K-\bar{k})u
,\;
u(0,\cdot) = u_{0},
\end{equation}

where $\bar{k}=\bar{k}_{u}=\int Kd\mu_{g_u}/\int d\mu_{g_u}$. We first check
\begin{lemma}\label{lem_second_flow_properties}
Along a flow line $u$, generated by \eqref{flow_2}, we have
\begin{enumerate}[label=(\roman*)]
\item
conservation of the volume, i.e.
$\partial_{t}\int  u^{\frac{2n}{n-2}}d\mu_{g_0}=0$
\item
preservation of positivity, precisely
$$ \min_{M}u_{0} \cdot
e^{(\inf_{M} K)t}
\leq  u(t) \leq
\max_{M}u_{0} \cdot e^{(\sup_{M} K-\inf_{M} K)t}.$$
\end{enumerate}
Moreover for every $L>0$ there exist $\gamma_{0},\delta_{0}>0$ such, that for every flow line $u$
on $\{ -k \leq  \gamma_{0} \} \cap \{ J\leq L \}$, generated by \eqref{flow_2}, there holds
$$
\begin{matrix*}[l]
(1)  &  \partial_{t}J(u)<-\delta_{0}   \\
(2) &  \partial_{t}(-k_{u})<-\delta_{0} \\
(3) & 0 <-r_{u} \neq o(1).
\end{matrix*}
$$
In particular $u$ leaves
$\{ -k\leq \gamma_{0} \} \cap \{ J\leq L \}$
through $\{ k=0 \}$, hitting
$ \{ k=0 \}$ transversally with $r_{u}<0$ and $ J(u)=0$.
\end{lemma}
\begin{proof}
(i) and (ii) are evident. Moreover from \eqref{partial_J} and \eqref{flow_2} we have
\begin{equation*}
\begin{split}
\partial_{t} J(u)
= \; &
\frac{2^{*}}{(-r_{u})^{\frac{n}{n-2}}}
\int(\frac{-k_{u}}{-r_{u}}\, L_{g_0}u-Ku^{\frac{n+2}{n-2}})(K-\bar{k}_{u})ud\mu_{g_{0}} \\
\leq & \;
- \frac{2^{*}}{(-r_{u})^{\frac{n}{n-2}}}
(\int K^{2}u^{\frac{2n}{n-2}}d\mu_{g_{0}} + O( \vert \frac{-k_{u}}{-r_{u}}\vert)),
\end{split}
\end{equation*}
using $\bar{k}_{u}=k_{u}$ due to $\int d\mu_{g_u}=\Vert u \Vert_{L^{\frac{2n}{n-2}}}^{\frac{2n}{n-2}}=1$.
Thus for  $u\in \{ -k \leq \gamma_{0} \} \cap \{ J\leq L \}  $
$$
\frac{-k_{u}}{-r_{u}}=J^{\frac{n-2}{n}}(u)(-k_{u})^{\frac{2}{n}} \leq L^{\frac{n-2}{n}}\gamma_{0}^{\frac{2}{n}}
$$
and we obtain
\begin{equation}\label{gra_fun}
\begin{split}
\partial_{t} J(u)
\leq & \;
- \frac{2^{*}}{(-r_{u})^{\frac{n}{n-2}}}
(\int K^{2}u^{\frac{2n}{n-2}}d\mu_{g_{0}} + O( L^{\frac{n-2}{n}}\gamma_{0}^{\frac{2}{n}})).
\end{split}
\end{equation}
Since  $0<-r_{u}$ is uniformly bounded from above on $X$
due to \eqref{equ_norm}, assertion (1) follows from \eqref{K_sqared_control}.
Furthermore and still on $\{ -k \leq \gamma_{0}\} \cap \{ J\leq L \}$
\begin{equation}\label{gra_k}
\begin{split}
\partial_{t} k_{u}
= \; &
\frac{2n}{n-2}\int Ku^{\frac{n+2}{n-2}}\partial_{t}ud\mu_{g_0}
=
\frac{2n}{n-2}\int Ku^{\frac{n+2}{n-2}}(K-\bar{k})u d\mu_{g_0}\\
\geq  \; &
\frac{2n}{n-2}\int Ku^{\frac{2n}{n-2}}d\mu_{g_{0}} + O(\gamma_{0}^2),
\end{split}
\end{equation}
whence also (2) is evident from \eqref{K_sqared_control}.
As a consequence of (1) and (2)
a flow line $u$ on $\{ -k\leq \gamma_{0} \} \cap \{ J \leq L \}$
has to leave that set
by hitting $\{ k=0 \}$ transversally at some time, say at $t=t_{1}$,
and necessarily $-k,-r>0$ during $[0,t_{1})$.
We may therefore assume, arguing by contradiction and contrarily to (3),
that
$$-k_{u},-r_{u}\xlongrightarrow{t \to t_{1}} 0.$$
Readily $\vert \partial_{t}r_{u} \vert \leq C $ by \eqref{equ_norm}, whence
$
\vert r_{u} \vert
\leq
C\vert t_{1}-t \vert
$
for $0\leq t \leq t_{1}$
and therefore
$$
\vert t_{1}-t \vert^{\frac{n}{n-2}}
\geq
c (- r_{u} )^{\frac{n}{n-2}}
=
c\frac{-k_{u}}{J(u)}.
$$
On the other hand from $\partial_{t}k_{u}>\delta_{0}$ we infer
$$
-k_{u}=\int^{t_{1}}_{t}\partial_{s}k_{u(s)}ds\geq \delta_{0}\vert t_{1}-t \vert,
$$
whence due to $u \in \{ J\leq L \} $ we conclude
\begin{equation}\label{finite_time_contradiction}
\vert t_{1}-t \vert^{\frac{n}{n-2}}
\geq
c\frac{-k_{u}}{J(u)}
\geq
\frac{c\delta_0}{L}\vert t_{1}-t \vert,
\end{equation}
a contradiction for $t_1\geq t\longrightarrow  t_1$.
\end{proof}

From Lemma \ref{lem_second_flow_properties} naturally
$$
E=\{ k=0 \} \cap \{ r<0 \} \cap \{ u>0 \} \cap \{ \Vert u \Vert_{L^{\frac{2n}{n-2}}}=1 \}  \subset C^{\infty}(M)
$$

comes into play, which we call the exit set, as justified by

\begin{proposition}\label{prop_exit_set_retraction}
If $\{ \partial J=0 \}= \emptyset$,
then for every $L>0$ the augmented sublevel
$$\{ J\leq L \} \cup E$$
retracts by strong deformation onto $E$ and
$$\{ J\leq L \} \cup E \simeq \{ J\leq L \} $$
are homotopically equivalent.
\end{proposition}	
\begin{proof}
First note, that due to $\{ \partial J=0 \}= \emptyset$ every flow line $u$,
generated by \eqref{flow_for_J} and starting at $u_{0}\in\{ J\leq L \} $,
has to tend to leave  $\{ J \leq L\} $,
which necessitates $0<-k_{u}\longrightarrow 0$.
This allows us to combine the movements along
\eqref{flow_for_J} and \eqref{flow_2} to a flow on $D=\{ J\leq L \} \cup E$
by choosing
$\gamma_{0},\delta_{0}>0$ such, that the conclusions of
Lemmata \ref{transverse} and \ref{lem_second_flow_properties}
hold true, and
\begin{enumerate}[label=(\roman*)]
\item[(a)] for $u_{0} \in D_{1}=\{ -k > \gamma_{0} \}\cap \{ J\leq L \} $
flow along \eqref{flow_for_J} and stay in $D_{1}$
until we hit $\{ -k=\gamma_{0} \} $ transversally,
which, as noted above, has to happen
\item[(b)] for $u_{0}\in D_{2}=\{ 0< -k \leq \gamma_{0} \} \cap \{ J\leq L \}$
flow along \eqref{flow_2} and stay in $D_{2}$ until we hit
$E=\{ k=0 \}\cap \{ r<0 \}  $ transversally
\item[(c)] 	for $u_{0}\in D_{3}=E$ we do not move.
\end{enumerate}
Readily this gives rise to a homotopy
$$
H_{1}\; : \; (\{ J\leq L \} \cup E) \times [0,1]\longrightarrow \{ J\leq L \} \cup E
$$
inducing a strong deformation retract
$\{ J\leq L \} \cup E \xhookrightarrow{\;sdr\:} E$ and we note, that
$$
H_{1}(\{ J\leq L \},1 ) = E.
$$
On the other hand we consider for $0<\varepsilon \ll 1$
$$
S_{\varepsilon}
=
\left(
\{ J<\frac{L}{2}\}
\cap
\{ \frac{-k}{-r}<\varepsilon \}
\cap
\{-k < \varepsilon \}
\right)
\cup
E
\subset
\{ J\leq L \} \cup E
$$
and, given $u_{0}\in S_{\varepsilon}$ as an initial datum,
solve inversely to \eqref{flow_2}
\begin{equation}\label{flow_inver}
\partial_t\tilde{u} = -(K-\bar{k}_{\tilde{u}})\tilde{u},\;
\tilde{u}(0,\cdot) = u_{0}.
\end{equation}
Using again \eqref{K_sqared_control},
we find constants $\delta,D>0$ such, that, as long as $\tilde u \in S_{\varepsilon}$,
$$
\begin{matrix*}[l]
(1) & |\partial_{t}r_{\tilde u}|<D \\
(2) & \partial_{t} (-k_{\tilde u})>\delta \\
(3) & \partial_{t}\frac{-k_{\tilde u}}{-r_{\tilde u}}>\delta \\
(4) & \partial_{t}J(\tilde u)>\delta,
\end{matrix*}
$$
cf. \eqref{equ_norm}, \eqref{gra_fun}, \eqref{gra_k}.
So $\tilde u$ will in finite time leave $S_{\varepsilon}$ transversally through
$$
T_{\varepsilon}
=
\{ J=\frac{L}{2}\}
\cup
\left(\{\frac{-k}{-r}=\varepsilon\}
\cup
\{-k = \varepsilon \}
\right)
\cap
\{J\leq L\}
\subset
\{J\leq L\},
$$
since by (1), (2), (3), if $0>r_{\tilde u}\longrightarrow 0$, then $\tilde{u}$ must
hit $\{ \frac{-k}{-r} =\varepsilon\}$ transversally, before reaching $r_{\tilde u}=0$.
Letting thus
$$
A_\varepsilon=(\{ J\leq L \}  \setminus S_{\varepsilon})\cup T_{\varepsilon}\subset\{J\leq L\}
$$
and for $u_{0}\in \{J\leq L \} \cup E$
$$0\leq t_{u_{0}}=\inf\{ t>0\; : \; \tilde u(t,\cdot) \in A_\varepsilon\}<\infty,$$
which by transversality depends continuously on $u_{0}$, we find, that under
$$
H_{2}:(\{ J\leq L \}\cup E )\times[0,1]\longrightarrow (\{ J\leq L \}\cup E )
: (u_{0},\tau)\longrightarrow \tilde u(\tau t_{u_{0}},\cdot)
$$
both $\{ J\leq L \} $ and $\{ J\leq L \} \cup E$
retract by strong deformation onto $A_{\varepsilon}$.
\end{proof}	
In particular $E\neq \emptyset$, if $\{ \partial J=0 \}=\emptyset$.
On the other hand $E=\emptyset$ has a much stronger impact
and relates closely to the the situation studied in \cite{Mayer_Zhu_Negative_Yamabe_1}.
\begin{corollary}\label{cor_exit_set_versus_A_B_inequality}
The assertions
\begin{enumerate}[label=(\roman*)]
\item a global A-B-inequality holds
\item on every sublevel an A-B-inequality holds
\item on some sublevel an A-B-inequality holds
\item $J$ admits a global minimizer on $X$
\item the exit set $E$ is empty
\end{enumerate}
are related by
$$
(i) \Longrightarrow (ii)  \Longleftrightarrow (iii) \Longleftrightarrow (iv) \Longleftrightarrow (v).
$$
\end{corollary}	
\begin{proof}
$(i) \Longrightarrow (ii)  \Longrightarrow (iii)$ are clear and for
$(iii) \Longrightarrow (iv)$ see \cite{Mayer_Zhu_Negative_Yamabe_1}.
Moreover, if $(iv)$ holds, then $\inf_{X} J>0$ by definition, while $\inf_{X} J=0$ via \eqref{flow_inver}, if $E \neq \emptyset$. Thus
$
(i) \Longrightarrow (ii)  \Longrightarrow (iii) \Longrightarrow (iv) \Longrightarrow (v)
$
and we are left with proving
$
(v)  \Longrightarrow  (ii).
$
From Lemma \ref{lem_second_flow_properties} we infer
$$
E=\emptyset
 \Longrightarrow
\; \forall \; L>0 \; \exists \; \gamma>\gamma_0
\; : \;
-k>\gamma \; \text{ on } \; \{ J\leq L \},
$$
which implies the validity of an A-B-inequality on $\{ J\leq L \} $.
\end{proof}	
We describe sublevels as strong deformation retracts of $X$.
\begin{lemma}\label{lem_sublevel_retracts}
If $\{ \partial J=0\}=\emptyset$, then $X$ retracts for every $L>0$ onto $\{ J\leq L \} $
by strong deformation.
\end{lemma}	
\begin{proof}
For $L>0$ and $u_{0} \in \{ J>L \}$
consider the flow line $u$,
generated by \eqref{flow_for_J} and
starting at $u_{0}$. Then $u$ hits $\{ J=L \} $ transversally
or, along a sequence in time,
$-k_{u}\longrightarrow 0$
or
$-r_{u}\longrightarrow 0$.
We claim, that necessarily $u$ hits $\{ J=L \} $ transversally
and thus have to exclude
$-k_{u}\longrightarrow 0$ or $-r_{u}\longrightarrow 0$ on $\{J>L\}$,
which, arguing by contradiction, we assume.
Then $-k_{u}\longrightarrow 0
 \Longrightarrow -r_{u}\longrightarrow 0$, as $J(u)> L$,
and due to $J(u)\leq J(u_{0})$
$$
-r_{u}\longrightarrow 0
\Longrightarrow
\frac{-k_{u}}{-r_{u}}\longrightarrow 0
 \Longrightarrow -k_{u}\longrightarrow 0,
$$
whence
\begin{equation}\label{equivalence_of_vanishing}
-k_{u}\longrightarrow 0
\Longleftrightarrow
-r_{u} \longrightarrow 0
\Longleftrightarrow
\frac{-k_{u}}{-r_{u}} \longrightarrow 0.
\end{equation}
We may therefore assume, that for every $\varepsilon>0$ after some time
$$
u \in S_{\varepsilon}=\{ 0<-r,-k,\frac{-k}{-r}\leq \varepsilon \}
\subset X.
$$
Let us also denote
$$t_{1}
=
\inf\{t> 0 \; : \;
-k_{u_{t_{n}}}\longrightarrow 0 \; \text{ for some } \;
0<t_{n}\nearrow t
\} \in (0,\infty].$$
Using \eqref{K_sqared_control} we then find, that on $S_{\varepsilon}$
along \eqref{flow_for_J}
$$
\begin{matrix*}[l]
(1) & \partial_{t} (-k_u) \leq -\delta \\	
(2) & \partial_{t}(-r_u) \geq -D \\
(3) &\partial_{t} \frac{-k_u}{-r_u}<-\delta
\end{matrix*}
$$
for some $\delta,D>0$.
We claim, that $u\in S_{2 \varepsilon}$ eventually, i.e.
$$
\; \exists \; 0\leq t_{0} < t_{1}
\; \forall \; t_{0}\leq t < t_{1}
\; : \;
u\in S_{2 \varepsilon}.
$$
In fact, as long as
$u\in S_{2 \varepsilon}$, by (1) and (3) necessarily
$$-k_u,\frac{-k_u}{-r_u}<\varepsilon,$$
whence due to \eqref{equivalence_of_vanishing}
the only possibility to leave $S_{2\varepsilon}$,
before reaching $\{k=0\}$,
is by \textit{increasing} $-r$ from $\varepsilon$ to $2\varepsilon$,
which due to (3) and
$$
J(u)=\frac{-k_{u}}{-r_{u}}\frac{1}{(-r_{u})^{\frac{2}{n-2}}}
$$
comes at an energetic cost.
Thus $u$ can travel at most finitely many times
from $S_{\epsilon}$ to $S_{2\epsilon}^{c}$ and the claim follows.
So we may assume, that (1), (2) and (3) hold during $[t_{0},t_{1})$,
whence (i) implies $t_{1}<\infty$.
Arguing as for \eqref{finite_time_contradiction},
we then reach the desired contradiction
and conclude, that every flow line $u$,
starting at $u_{0}$ with $J(u_{0})>L$,
hits $\{ J=L \} $ transversally at some time $t_{u_{0}}$,
depending continuously on $u_{0}$ by transversality.
Letting
$t_{u_{0}}=0 \; \text{ for } \; u_{0}\in \{ J\leq L \}$,
we find the desired strong deformation retract as
$$
H:X\times [0,1] \longrightarrow X:
( u_{0},\tau)\longrightarrow  u(\tau t_{u_{0}},\cdot).
\qedhere
$$
\end{proof}

\begin{lemma}\label{lem_homotopy_on_X}
The map
$
H:X \times [0,1] \longrightarrow X:(u,\tau)\longrightarrow
u_{\tau}
=
\frac{w_{\tau}}{\Vert w_{\tau} \Vert_{L^{\frac{2n}{n-2}}}}
$
with
$$
w_{\tau}=(\tau + (1-\tau)u^{\frac{2n}{n-2}})^{\frac{n-2}{2n}}
$$
defines a null homotopy.
\end{lemma}	
\begin{proof}
In view of \eqref{definition_of_X} we have to verify
$r_{w_{\tau}},k_{w_{\tau}}<0$. Clearly
\begin{equation*}
k_{w_{\tau}}
=
\tau k_{1} + (1-\tau)k_{u}
<0
\end{equation*}
and $r_{w_{\tau}}<0$ for $\tau \in \{ 0,1 \}$, while from
\begin{equation*}
\begin{split}
r_{w_{\tau}}
= \; &
\int c_n\vert \nabla (\tau + (1-\tau)u^{\frac{2n}{n-2}})^{\frac{n-2}{2n}} \vert^{2}
-
\vert (\tau + (1-\tau)u^{\frac{2n}{n-2}})^{\frac{n-2}{2n}} \vert^{2}d\mu_{g_{0}} \\
= \;&
(1-\tau)^{2} c_n\int \vert \nabla u \vert^{2}\vert \frac{u^{\frac{2n}{n-2}}}{\tau + (1-\tau)u^{\frac{2n}{n-2}}} \vert^{\frac{n+2}{n}}
d\mu_{g_{0}} \\
& -
\int \vert \tau + (1-\tau)u^{\frac{2n}{n-2}} \vert^{\frac{n-2}{n}}d\mu_{g_{0}}
\end{split}
\end{equation*}
we find, that
\begin{equation*}
r_{w_{\tau}}
\leq
(1-\tau)^{\frac{n-2}{n}}c_n\int \vert \nabla u \vert^{2}d\mu_{g_{0}}
-
(1-\tau)^{\frac{n-2}{n}}\int u^2d\mu_{g_{0}}
=
(1-\tau)^{\frac{n-2}{n}}r_{u},
\end{equation*}
whence due to $r_{u}<0$ also $r_{w_{\tau}}<0$ for $\tau \in (0,1)$.
The claim follows.
\end{proof}
We can finally state and prove the main result of this section.
\begin{corollary}\label{cor_E_contractible}
If $\{ \partial J=0 \}= \emptyset$, then $E$ is contractible.
\end{corollary}	
\begin{proof}
$X$ is  contractible by Lemma \ref{lem_homotopy_on_X}
and so are sublevels $\{ J\leq L \} $
by virtue of Lemma \ref{lem_sublevel_retracts}.
Proposition \ref{prop_exit_set_retraction} then shows, that
$\{ J\leq L \} \,\cup\, E$ is contractible as well and
retracts by deformation onto $E$.
\end{proof}

\begin{proof}[\textbf{Proof of Theorem \ref{thm_exit_set_example}, part (i) and (ii)}]
For (i) see Corollary \ref{cor_exit_set_versus_A_B_inequality} and for (ii) confer
Corollary \ref{cor_E_contractible} and its proof.
\end{proof}
For part (iii) of Theorem \ref{thm_exit_set_example}, see Proposition \ref{prop_disconnectedness_of_E_for_K_dp}.

\section{A Non Connected  Exit Set}\label{sec_non_connectedness}

As discussed in \cite{Mayer_Zhu_Negative_Yamabe_1}, $J$ does in absence of solutions to
$\partial J=0$ not exhibit critical points at infinity and hence
the latter cannot be used to prove existence of the former.
On the other hand and in view of Corollary \ref{cor_E_contractible}
a Morse theoretical study of $E$ looks promising
to reach a contradiction to the contractibility of $E$ and thereby showing
solvability of $\partial J=0$ - for instance by studying the energy
$$\mathcal{E}: E\longrightarrow\mathbb{R}:\ u \longrightarrow  -r_u.$$
\indent
Here we limit ourselves to the construction of a specific function $K=K_{dp}$ with non contractible exit set $E$.
To this end let
\begin{equation*}
\bar \theta_{a,\lambda}=\frac{1}{1+\lambda^{2}\gamma_{n} G_{a}^{\frac{2}{2-n}}}
\; \text{ on } \;
\{ G_{a}>0 \}
\; \text{ for } \;
\lambda>0
\end{equation*}
and
$
\bar\varphi_{a, \lambda }
=
\eta_{a}\bar\theta_{a,\lambda}
$
with $\eta_{a}$ as in \eqref{cut_off_function}, cf. \eqref{Bubble_Definition}.
In what follows, consider for
$$
\bar{\lambda}_1,\bar{\lambda}_{2} \gg 1
\; \text{ and } \;
\bar{a}_1,\bar{a}_{2}\in M
\; \text{ with } \;
\text{dist}(\bar{a}_1, \bar{a}_2)>4\epsilon
$$
the double peak type function
\begin{equation}\label{K_in_the_exit_set_example}
K=K_{dp}=-\bar{\alpha}+\bar{\varphi}_1+\bar{\varphi}_2,\; \;
0\leq \bar{\varphi}_i={\eta}_{\bar{a}_i}\bar\theta_{\bar{a}_i,\bar{\lambda}_i} \leq 1.
\end{equation}
We will show, that
for $\bar{\lambda}_1$, $\bar{\lambda}_2$ sufficiently large
and
\begin{equation*}
 \bar{\alpha}=\frac{c_{1}}
{(\sfrac{4n(n-1)c_{1}}{\vert M \vert^{\frac{2}{n}} })^{\frac{n}{n-2}}+c_{1}}
\end{equation*}
there are at least two distinct connected components
$E_{1},\; E_{2}\subset U(1, \varepsilon)$ of $E$,
and in particular $E$ is not contractible, cf. Definition \ref{V_p_e} and \eqref{bar_alpha}.

\begin{proposition}\label{prop_exp}
For $K=K_{dp}$ and
$$u=\alpha+\alpha_1\varphi_{1}+v\in U(1, \varepsilon)\cap\{d(a_1, \bar{a}_1)<\varepsilon\}$$
up to some
$$
o_
{
\bar{\lambda}_{1}d(\bar{a}_{1},a_{1})+\sfrac{1}{\lambda_{1}}+\Vert v \Vert
+
\sum_{i}\sfrac{1}{\bar{\lambda}_{i}}
}
(
\bar{\lambda}_{1}^{2}d^{2}(\bar{a}_{1},a_{1})+\lambda_{1}^{\frac{2-n}{2}}+\frac{\bar{\lambda}_{1}^{2}}{\lambda_{1}^{2}}+\Vert v \Vert^{2}
)
$$
we have
\begin{flalign}
(i) \quad
1=\; &\|u\|_{L^{\frac{2n}{n-2}}}^{\frac{2n}{n-2}} &\notag\\
= \; &
|M|\alpha^{\frac{2n}{n-2}}+c_1\alpha_1^{\frac{2n}{n-2}} + \frac{2n}{n-2}\alpha^{\frac{n+2}{n-2}}\alpha_1\int \varphi_{1}d\mu_{g_0} & \label{norm} \\
& + \frac{2n}{n-2}b_0\frac{ \alpha\alpha_1^{\frac{n+2}{n-2}}}{\lambda_1^{\frac{n-2}{2}}}
+
\frac{n(n+2)}{(n-2)^2}\int (\alpha^{\frac{4}{n-2}}+\alpha_1^{\frac{4}{n-2}}\varphi_{1}^{\frac{4}{n-2}}) v^2d\mu_{g_0}
& \notag
\end{flalign}
\begin{flalign}
(ii) \quad
r_u
= \;&
-|M|\alpha^2 + 4n(n-1)c_1\alpha_1^2 - 2\alpha\alpha_1\int \varphi_{1}d\mu_{g_0}  + \int L_{g_0}vvd\mu_{g_0}
\label{exp_r} &
\end{flalign}
\begin{flalign}
(iii) ~~
k_u
= &
-\bar{\alpha}
+
O^+(\bar{\lambda}_i^{-2})
+
c_1\alpha_1^{\frac{2n}{n-2}}
-
c_1\alpha_1^{\frac{2n}{n-2}}\bar{\lambda}_1^2d^2(\bar{a}_1, a_1) \label{exp_k}
\\
& -
c_4\alpha_1^{\frac{2n}{n-2}} \frac{\bar{\lambda}_1^2}{\lambda_1^2}
+
\frac{2n}{n-2}b_0\frac{ \alpha\alpha_1^{\frac{n+2}{n-2}}}{\lambda_1^{\frac{n-2}{2}}}
+
\frac{n(n+2)}{(n-2)^2}\alpha_1^{\frac{4}{n-2}}\int v^2\varphi_{1}^{\frac{4}{n-2}}d\mu_{g_0},
   & \notag
\end{flalign}
where
$ c_4=\int_{\mathbb{R}^n}\frac{|x|^2}{(1+|x|^2)^n}dx$ and $c_1, b_0>0$ are as in Lemma \ref{lem_interactions}.
\end{proposition}

We postpone the proof of Proposition \ref{prop_exp} to the appendix.
Using these expansions, we now prove non connectedness of the exit set $E$, related to the function $K=K_{dp}$ to be prescribed, which readily implies (iii) of Theorem \ref{thm_exit_set_example}.

\begin{proposition}\label{prop_disconnectedness_of_E_for_K_dp}
The exit set $E$ of $K_{dp}$ has at least two connected components.
\end{proposition}
\begin{proof}
We first consider for some $0<\varepsilon \ll 1$
\begin{equation}\label{set_u}
\begin{split}
u \in
\bar{U}_{\bar{a}_{1}}(1,\varepsilon)
=
U(1, \varepsilon)  \cap \{d(a_1, \bar{a}_1)\leq\varepsilon\}
\cap \{r\leq0\}
\subset C^{\infty}(M),
\end{split}
\end{equation}
writing $u=\alpha+\alpha_1\varphi_{1}+v$ accordingly.
Hence by Proposition \ref{prop_exp} and, as
$$0<-r<C\ \;\text{ on }\; \{r<0\}\cap\{\|\cdot\|_{L^{\frac{2n}{n-2}}}=1\},$$
we have uniform bounds
\begin{equation}\label{bound}
  0< c< \alpha,\ \alpha_1,\ \|u\|<C<\infty.
\end{equation}
Note, that in view of \eqref{norm} and \eqref{exp_r} of Proposition \ref{prop_exp}, fixing $0<\tau\ll1$ and $r=-\tau$,
we may by means of the implicit function theorem uniquely
determine
\begin{equation*}
\alpha=\alpha(a_{1}, \lambda_{1}, v),\; \alpha_1=\alpha_1(a_{1}, \lambda_{1}, v) \; \text{ on } \;  \{r=-\tau\}.
\end{equation*}
Then, in order to proceed with the relevant expansions,
as an intermediate step we deduce from Proposition \ref{prop_exp}, that up to some
$$
O
(
\bar{\lambda}_{1}^{2}d^{2}(\bar{a}_{1},a_{1})+\lambda_{1}^{\frac{2-n}{2}}+\frac{\bar{\lambda}_{1}^{2}}{\lambda_{1}^{2}}+\Vert v \Vert^{2}
)
$$
for $u \in \bar{U}_{\bar{a}_{1}}(1,\varepsilon)$ there holds
\begin{flalign}
(i) \quad &
1
=
\vert M \vert \alpha^{\frac{2n}{n-2}}
+
c_{1}\alpha_{1}^{\frac{2n}{n-2}}
\label{rough_expansion_normalisation} \\
(ii) \quad &
r
=
- \vert M \vert \alpha^{2}+4n(n-1)c_{1}\alpha_{1}^{2}
\label{rough_expansion_r} \\
(iii) \quad &
k
=
-
\bar{\alpha}
+
O^+(\bar{\lambda}_i^{-2})
+
c_{1}\alpha_{1}^{\frac{2n}{n-2}}.
\label{rough_expansion_k}		
\end{flalign}	
Hence, in view of \eqref{rough_expansion_k} and in order for $k=k_{u}$ to satisfy
\begin{equation}\label{k_roughly_large_condition}
k
=
\sup_{ \bar{u}\in \bar{U}_{\bar{a}_{1}}(1,\varepsilon) } k_{\bar{u}}
+
O
(
\bar{\lambda}_{1}^{2}d^{2}(\bar{a}_{1},a_{1})+\lambda_{1}^{\frac{2-n}{2}}+\frac{\bar{\lambda}_{1}^{2}}{\lambda_{1}^{2}}+\Vert v \Vert^{2}
),
\end{equation}
due to \eqref{rough_expansion_normalisation} and \eqref{rough_expansion_r}
necessarily
$$
0<-r=-r_{u}=
O
(
\bar{\lambda}_{1}^{2}d^{2}(\bar{a}_{1},a_{1})+\lambda_{1}^{\frac{2-n}{2}}+\frac{\bar{\lambda}_{1}^{2}}{\lambda_{1}^{2}}+\Vert v \Vert^{2}
)
$$
and the $(\alpha,\alpha_{1})$-variables of
$u=\alpha+\alpha_1\varphi_{1}+v$
are  up to some
$$O
(
\bar{\lambda}_{1}^{2}d^{2}(\bar{a}_{1},a_{1})+\lambda_{1}^{\frac{2-n}{2}}+\frac{\bar{\lambda}_{1}^{2}}{\lambda_{1}^{2}}+\Vert v \Vert^{2}
)
$$
determined by the relations
$$
1=\vert M \vert \alpha^{\frac{2n}{n-2}}+c_{1}\alpha_{1}^{\frac{2n}{n-2}}
\; \text{ and } \;
r=-\vert M \vert \alpha^{2}+4n(n-1)c_{1} \alpha_{1}^{2},
$$
in particular,
\begin{equation*}
\begin{split}
1
= \; &
\vert M \vert
(\frac{4n(n-1)c_{1}\alpha_{1}^{2}-r}{\vert M \vert })^{\frac{n}{n-2}}
+
c_{1}\alpha_{1}^{\frac{2n}{n-2}}
=
[(\frac{4n(n-1)c_{1}}{\vert M \vert^{\frac{2}{n}} })^{\frac{n}{n-2}}+c_{1}]
\alpha_{1}^{\frac{2n}{n-2}}.
\end{split}
\end{equation*}
We conclude, that up to some
$O(
\bar{\lambda}_{1}^{2}d^{2}(\bar{a}_{1},a_{1})+\lambda_{1}^{\frac{2-n}{2}}+\frac{\bar{\lambda}_{1}^{2}}{\lambda_{1}^{2}}+\Vert v \Vert^{2}
)$
\begin{flalign}
1) \quad &
\alpha_{1}
=
\beta_{1}
=
[
(\frac{4n(n-1)c_{1}}{\vert M \vert^{\frac{2}{n}} })^{\frac{n}{n-2}}
+
c_{1}
]^{\frac{2-n}{2n}}
\label{value_beta_1} \\
2) \quad &
\alpha
=
\beta
=
(\frac{4n(n-1)c_{1}}{\vert M \vert })^{\frac{1}{2}}
\beta_{1}.
\label{value_beta}
\end{flalign}	
Plugging \eqref{value_beta_1},\eqref{value_beta} into Proposition \ref{prop_exp}
we then find up to some
$$
o_
{
\bar{\lambda}_{1}d(\bar{a}_{1},a_{1})+\sfrac{1}{\lambda_{1}}+\Vert v \Vert
+
\sum_{i}\sfrac{1}{\bar{\lambda}_{i}}
}
(
\bar{\lambda}_{1}^{2}d^{2}(\bar{a}_{1},a_{1})+\lambda_{1}^{\frac{2-n}{2}}+\frac{\bar{\lambda}_{1}^{2}}{\lambda_{1}^{2}}+\Vert v \Vert^{2}
)
$$
for $u \in \bar{U}_{\bar{a}_{1}}(1,\varepsilon)$
satisfying \eqref{k_roughly_large_condition}
the simplified expansions
\begin{flalign}
(i) \quad
1
= &
\vert M \vert \alpha^{\frac{2n}{n-2}}
+
c_{1}\alpha_{1}^{\frac{2n}{n-2}}
+
\frac{2n}{n-2}\beta^{\frac{n+2}{n-2}}\beta_{1}\int \varphi_{1}d\mu_{g_0}
&\label{faitful_condition_on_normalization}\\
&
+
 \frac{2n}{n-2}b_0\frac{\beta \beta_{1}^{\frac{n+2}{n-2}}}{\lambda_{1}^{\frac{n-2}{2}}}
+
\frac{n(n+2)}{(n-2)^{2}}
\int (\beta^{\frac{4}{n-2}}+\beta_{1}^{\frac{4}{n-2}}\varphi_{1}^{\frac{4}{n-2}})v^{2}d\mu_{g_{0}}
&\notag
\end{flalign}	
\begin{flalign}
(ii) \quad
r
= &
- \vert M \vert \alpha^{2}+4n(n-1)c_{1}\alpha_{1}^{2}
-
2\beta \beta_{1}\int \varphi_1 d\mu_{g_0}
+
\int L_{g_{0}}vv d\mu_{g_{0}}
&\label{faitful_condition_on_the_expansion_of_r}
\end{flalign}
\begin{flalign}
(iii) \quad
k
= &
-\bar{\alpha}
+
O^+(\bar{\lambda}_i^{-2})
+
c_{1}\alpha_{1}^{\frac{2n}{n-2}}
-
c_{1}\beta_{1}^{\frac{2n}{n-2}}\bar{\lambda}_{1}^{2}d^{2}(a_{1},\bar{a}_{1})
\label{faitful_expansion_on_k} \\
& -
c_{4}\beta_{1}^{\frac{2n}{n-2}}\frac{\bar{\lambda}_{1}^{2}}{\lambda_{1}^{2}}
+
\frac{2n}{n-2}
b_0\frac{\beta\beta_{1}^{\frac{n+2}{n-2}}}{\lambda_{1}^{\frac{n-2}{2}}} +
\frac{n(n+2)}{(n-2)^{2}}\beta_{1}^{\frac{4}{n-2}}
\int \varphi_{1}^{\frac{4}{n-2}}v^{2}d\mu_{g_{0}}.
&\notag
\end{flalign}
Next, still assuming \eqref{k_roughly_large_condition}, from \eqref{faitful_condition_on_the_expansion_of_r} and up to some
\begin{equation}\label{exit_set_example_general_negligible_error_term}
o_
{
-r+\bar{\lambda}_{1}d(\bar{a}_{1},a_{1})+\sfrac{1}{\lambda_{1}}+\Vert v \Vert
+
\sum_{i}\sfrac{1}{\bar{\lambda}_{i}}
}
(
-r + \bar{\lambda}_{1}^{2}d^{2}(\bar{a}_{1},a_{1})+\lambda_{1}^{\frac{2-n}{2}}+\frac{\bar{\lambda}_{1}^{2}}{\lambda_{1}^{2}}+\Vert v \Vert^{2}
)
\end{equation}
we find
\begin{equation*}
\begin{split}
\alpha^{\frac{2n}{n-2}}
= \; &
\vert M \vert^{\frac{n}{2-n}}
[
4n(n-1)c_{1}\alpha_{1}^{2}-r - 2 \beta \beta_{1}\int\varphi_1 d\mu_{g_0} + \int L_{g_{0}}vv d\mu_{g_{0}}
]^{\frac{n}{n-2}} \\
= \; &
\vert M \vert^{\frac{n}{2-n}}
(4n(n-1)c_{1})^{\frac{n}{n-2}}\alpha_{1}^{\frac{2n}{n-2}} \\
& -
\frac{n}{n-2}
\vert M \vert^{\frac{n}{2-n}}
(4n(n-1)c_{1})^{\frac{2}{n-2}}\beta_{1}^{\frac{4}{n-2}}\\
& \quad \quad \quad \quad \quad \quad \quad \quad \quad
\cdot[r+2 \beta \beta_{1}\int\varphi_1 d\mu_{g_0} - \int L_{g_{0}}vv d\mu_{g_{0}}]
\end{split}
\end{equation*}
and plugging this into \eqref{faitful_condition_on_normalization}, that up to the same error \eqref{exit_set_example_general_negligible_error_term}
\begin{equation*}
\begin{split}
1
= \; &
\vert M \vert
\bigg(
\vert M \vert^{\frac{n}{2-n}}
(4n(n-1)c_{1})^{\frac{n}{n-2}}\alpha_{1}^{\frac{2n}{n-2}}  \\
&  \quad \quad \;\;-
\frac{n}{n-2}
\vert M \vert^{\frac{n}{2-n}}
(4n(n-1)c_{1})^{\frac{2}{n-2}}\beta_{1}^{\frac{4}{n-2}}\\
& \quad \quad \quad \quad \quad \quad \quad \quad \quad
\cdot
[r+2 \beta \beta_{1}\int \varphi_1 d\mu_{g_0}-\int L_{g_{0}}vv d\mu_{g_{0}}]
\bigg) \\
& +
c_{1}\alpha_{1}^{\frac{2n}{n-2}}
+
\frac{2n}{n-2}\beta^{\frac{n+2}{n-2}}\beta_{1}\int \varphi_1 d\mu_{g_0}
+
\frac{2n}{n-2}b_0\frac{\beta \beta_{1}^{\frac{n+2}{n-2}}}{\lambda_{1}^{\frac{n-2}{2}}}\\
&
+
\frac{n(n+2)}{(n-2)^{2}}
\int (\beta^{\frac{4}{n-2}}+\beta_{1}^{\frac{4}{n-2}}\varphi_{1}^{\frac{4}{n-2}})v^{2}d\mu_{g_{0}}
\\
= \; &
[(\frac{4n(n-1)c_{1}}{\vert M \vert^{\frac{2}{n}} })^{\frac{n}{n-2}}+c_{1}]\alpha_{1}^{\frac{2n}{n-2}}
-
\frac{n}{n-2}(\frac{4n(n-1)c_{1}}{\vert M \vert })^{\frac{2}{n-2}}\beta_{1}^{\frac{4}{n-2}}r\\
&
+
\frac{2n}{n-2}b_0\frac{\beta \beta_{1}^{\frac{n+2}{n-2}}}{\lambda_{1}^{\frac{n-2}{2}}}  +
\frac{n}{n-2}(\frac{4n(n-1)c_{1}}{\vert M \vert })^{\frac{2}{n-2}}\beta_{1}^{\frac{4}{n-2}}\int L_{g_{0}}vv d\mu_{g_{0}}\\
&
+
\frac{n(n+2)}{(n-2)^{2}}
\int (\beta^{\frac{4}{n-2}}+\beta_{1}^{\frac{4}{n-2}}\varphi_{1}^{\frac{4}{n-2}})v^{2}d\mu_{g_{0}}
\\
& +
\frac{2n}{n-2}\beta \beta_{1}
(\beta^{\frac{4}{n-2}}
-
(\frac{4n(n-1)c_{1}}{\vert M \vert })^{\frac{2}{n-2}} \beta_{1}^{\frac{4}{n-2}}
)\int\varphi_1d\mu_{g_0}
.
\end{split}
\end{equation*}
The latter term vanishes
due to \eqref{value_beta}, whence
up to an error as in \eqref{exit_set_example_general_negligible_error_term}
\begin{equation*}
\begin{split}
\alpha_{1}^{\frac{2n}{n-2}}
=\; &
[(\frac{4n(n-1)c_{1}}  {\vert M \vert^{\frac{2}{n}} } )^{\frac{n}{n-2}}  + c_{1}]^{-1}
[
1
 +
\frac{n}{n-2}(\frac{4n(n-1)c_{1}}{\vert M \vert })^{\frac{2}{n-2}}\beta_{1}^{\frac{4}{n-2}}r\\
&
-
\frac{2n}{n-2}b_0\frac{\beta \beta_{1}^{\frac{n+2}{n-2}}}{\lambda_{1}^{\frac{n-2}{2}}}
-
\frac{n}{n-2}(\frac{4n(n-1)c_{1}}{\vert M \vert })^{\frac{2}{n-2}}\beta_{1}^{\frac{4}{n-2}}\int L_{g_{0}}vv d\mu_{g_{0}} \\
&
-
\frac{n(n+2)}{(n-2)^{2}}
\int (\beta^{\frac{4}{n-2}}+\beta_{1}^{\frac{4}{n-2}}\varphi_{1}^{\frac{4}{n-2}})v^{2}d\mu_{g_{0}}
]
\end{split}
\end{equation*}
for $u \in \bar{U}_{\bar{a}_{1}}(1,\varepsilon)$ satisfying \eqref{k_roughly_large_condition},
and inserting this into \eqref{faitful_expansion_on_k} we find
\begin{equation*}
\begin{split}
k
= \; &
-\bar{\alpha}
+
O^+(\bar{\lambda}_i^{-2})
+
\frac{c_{1}}
{(\frac{4n(n-1)c_{1}}{\vert M \vert^{\frac{2}{n}} })^{\frac{n}{n-2}}+c_{1}}
\\ \quad &
\Bigg[
1
+
\frac{n}{n-2}(\frac{4n(n-1)c_{1}}{\vert M \vert })^{\frac{2}{n-2}}\beta_{1}^{\frac{4}{n-2}}r
-
\frac{2n}{n-2}b_0\frac{\beta \beta_{1}^{\frac{n+2}{n-2}}}{\lambda_{1}^{\frac{n-2}{2}}}
\\
&
\quad -
\frac{n}{n-2}(\frac{4n(n-1)c_{1}}{\vert M \vert })^{\frac{2}{n-2}}\beta_{1}^{\frac{4}{n-2}}\int L_{g_{0}}vv d\mu_{g_{0}} \\
& \quad
-
\frac{n(n+2)}{(n-2)^{2}}
\int (\beta^{\frac{4}{n-2}}+\beta_{1}^{\frac{4}{n-2}}\varphi_{1}^{\frac{4}{n-2}})v^{2}d\mu_{g_{0}}
\Bigg] \\
& -
c_{1}\beta_{1}^{\frac{2n}{n-2}}\bar{\lambda}_{1}^{2}d^{2}(\bar{a}_{1},\bar{a}_{1})
-
c_{4}\beta_{1}^{\frac{2n}{n-2}}\frac{\bar{\lambda}_{1}^{2}}{\lambda_{1}^{2}}
\\
& +
\frac{2n}{n-2}b_0\frac{\beta\beta_{1}^{\frac{n+2}{n-2}}}{\lambda_{1}^{\frac{n-2}{2}}}
+
\frac{n(n+2)}{(n-2)^{2}}\beta_{1}^{\frac{4}{n-2}}
\int \varphi_{1}^{\frac{4}{n-2}}v^{2}d\mu_{g_{0}},
\end{split}
\end{equation*}
whence from \eqref{value_beta} we derive after some simplifications
\begin{equation}\label{faithful_expansion_of_k_in_explicit_form}
\begin{split}
k = \;  &
-\bar{\alpha}
+
O^+(\bar{\lambda}_i^{-2})
+
\frac{c_{1}}
{(\frac{4n(n-1)c_{1}}{\vert M \vert^{\frac{2}{n}} })^{\frac{n}{n-2}}+c_{1}} \\
& +
\frac{n}{n-2}(\frac{4n(n-1)c_{1}}{\vert M \vert })^{\frac{2}{n-2}}
\frac{c_{1}\beta_{1}^{\frac{4}{n-2}}}
{(\frac{4n(n-1)c_{1}}{\vert M \vert^{\frac{2}{n}} })^{\frac{n}{n-2}}+c_{1}}
\cdot r
\\
& +
\frac{2n}{n-2}b_0
[1-
\frac{c_{1}}
{(\frac{4n(n-1)c_{1}}{\vert M \vert^{\frac{2}{n}} })^{\frac{n}{n-2}}+c_{1}}]
\frac{\beta \beta_{1}^{\frac{n+2}{n-2}}}{\lambda_{1}^{\frac{n-2}{2}}}\\
& -
\frac{n}{n-2}\frac{c_{1}\beta^{\frac{4}{n-2}}}
{(\frac{4n(n-1)c_{1}}{\vert M \vert^{\frac{2}{n}} })^{\frac{n}{n-2}}+c_{1}}
[
\int L_{g_{0}}vvd\mu_{g_{0}}
+
\frac{n+2}{n-2}\int v^{2}d\mu_{g_{0}} \\
& \quad\quad\quad\quad\quad\quad\quad\quad\quad\quad\quad\quad\quad
-
4n(n-1)\frac{n+2}{n-2}\int \varphi_{1}^{\frac{4}{n-2}}v^{2}d\mu_{g_{0}}
] \\
&
-
c_{0}\beta_{1}^{\frac{2n}{n-2}}\bar{\lambda}_{1}^{2}d^{2}(a_{1},\bar{a}_{i})
-
c_{2}\beta_{1}^{\frac{2n}{n-2}}\frac{\bar{\lambda}_{1}^{2}}{\lambda_{1}^{2}}.
\end{split}
\end{equation}
Since $L_{g_{0}}=-c_{n}\Delta_{g_{0}}-1$, where $c_{n}=4\frac{n-1}{n-2}$, we find with some constant $\gamma_{v}>0$
\begin{equation*}
\begin{split}
\int L_{g_{0}}vvd\mu_{g_{0}} &
+
\frac{n+2}{n-2}\int v^{2}d\mu_{g_{0}}
-
4n(n-1)\frac{n+2}{n-2}\int \varphi_{1}^{\frac{4}{n-2}}v^{2}d\mu_{g_{0}} \\
= \; &
c_{n} \int
[
-\Delta_{g_{0}}v
+
\frac{1}{n-1}v
-
n(n+2)
]
v d\mu_{g_{0}}
\geq
\gamma_{v}\Vert v \Vert^{2},
\end{split}
\end{equation*}
cf. Appendix D in \cite{Rey}. Thus, recalling \eqref{exit_set_example_general_negligible_error_term} and letting
\begin{equation}\label{shape_of_the_constant_term}
\gamma_{0}=-\bar{\alpha}
+
O^+(\bar{\lambda}_i^{-2})
+
\frac{c_{1}}
{(\frac{4n(n-1)c_{1}}{\vert M \vert^{\frac{2}{n}} })^{\frac{n}{n-2}}+c_{1}},
\end{equation}
with readily given constants $\gamma_{1},\gamma_{2},\gamma_{3}>0$
we  obtain from \eqref{faithful_expansion_of_k_in_explicit_form}
\begin{align}\label{faithful_expansion_of_k_in_explicit_form_1}
k = \;  &
\gamma_{0} + \gamma_{1}r - \gamma_{2}\bar{\lambda}_{1}^{2}d^{2}(a_{1},\bar{a}_{1})
+
\gamma_{3}\lambda_{1}^{\frac{2-n}{2}} - \gamma_{4} \vert \frac{\bar{\lambda}_{1}}{\lambda_{1}} \vert^{2}
-
O^{+}(\Vert v \Vert^{2}) \\
& +
o_
{
-r+\bar{\lambda}_{1}d(\bar{a}_{1},a_{1})+\sfrac{1}{\lambda_{1}}+\Vert v \Vert
+
\sum_{i}\sfrac{1}{\bar{\lambda}_{i}}
}
(
-r + \bar{\lambda}_{1}^{2}d^{2}(\bar{a}_{1},a_{1})+\lambda_{1}^{\frac{2-n}{2}}+\frac{\bar{\lambda}_{1}^{2}}{\lambda_{1}^{2}}+\Vert v \Vert^{2}
)\nonumber
\end{align}
for $u \in \bar{U}_{\bar{a}_{1}}(1,\varepsilon)$ satisfying \eqref{k_roughly_large_condition}.
In view of \eqref{shape_of_the_constant_term} let
\begin{equation}\label{bar_alpha}
 \bar{\alpha}=\frac{c_{1}}
{(\frac{4n(n-1)c_{1}}{\vert M \vert^{\frac{2}{n}} })^{\frac{n}{n-2}}+c_{1}}
\end{equation}
such, that $\gamma_{0}>0$ slightly positive. Then from \eqref{faithful_expansion_of_k_in_explicit_form_1} we easily see, that for
\begin{enumerate}[label=(\roman*)]
\item $n\geq 7$ we can readily ignore the $\lambda_{1}^{\frac{2-n}{2}}$-term
and find around
\begin{equation*}
  r=0,\; a_{1}=\bar{a}_{1},\; \lambda_{1}=\infty  \; \text{ and } \; v=0
\end{equation*}
a sign changing Morse type maximum structure of $k=k(r, a_{1}, \lambda_{1}, v)$, i.e.
$$
k>0 \;\text{ on }\; B^{-,+}_{\delta}
\; \text{ and } \;
k<0 \; \text{ on } \;  A^{-,+}_{D_{1},D_{2}}
$$
for suitable $0<\delta< D_{1} < D_{2} \ll 1$, a \textit{pointed quarter ball}
\begin{equation*}
\begin{split}
B^{-,+}_{\delta}
=
\{
(r&,\lambda_{1},a_{1},v)  \in \R \times \R  \times M \times W^{1,2}(M)
\; : \;\\
&
0
<
-\gamma_{1}	r
+
\gamma_{2}\bar{\lambda}_{1}^{2}d^{2}(a_{1},\bar{a}_{1})
+\gamma_{4}\vert \frac{\bar{\lambda}_{1}}{\lambda_{1}}\vert^{2}
+
O^{+}(\Vert v \Vert^{2})
< \delta^{2}
\}
 \end{split}
\end{equation*}
and a surrounding \textit{quarter annulus} $ A^{-,+}_{ D_{1},D_{2}}$ of type
\begin{equation*}
\begin{split}
A^{-,+}_{D_{1},D_{2}}
=
\{
( & r,\lambda_{1}, a_{1},v) \in \R \times \R \times M \times W^{1,2}(M)
\; : \;\\
&
D_{1}^{2}
<
-\gamma_{1}	r
+
\gamma_{2}\bar{\lambda}_{1}^{2}d^{2}(a_{1},\bar{a}_{1})
+\gamma_{4}\vert \frac{\bar{\lambda}_{1}}{\lambda_{1}}\vert^{2}
+
O^{+}(\Vert v \Vert^{2})
<D_2^{2}\}
  \end{split}
\end{equation*}
\item $n=6$ we argue as before, as we can ignore the $\lambda_{1}^{\frac{2-n}{2}}$-term due to $\bar{\lambda}_{1} \gg 1$
\item $3\leq n\leq 5$ we lose the maximum structure due to the $\lambda_{1}^{\frac{2-n}{2}}$-term, but
\begin{enumerate}[label=(\roman*)]
\item[1)] $k>0$ inside $B_{\delta}^{-,+}$, as $\gamma_{0}$ is sightly positive
\item[2)] $k<0$ on $A^{-,+}_{D_{1},D_{2}}$ by observing, cf. \eqref{faithful_expansion_of_k_in_explicit_form_1}, that
    \begin{equation*}
      \begin{split}
        \gamma_{3}\frac{1}{\lambda_{1}^{\frac{n-2}{2}}} - \gamma_{4} \vert \frac{\bar{\lambda}_{1}}{\lambda_{1}} \vert^{2} < 0
\Longleftrightarrow \;
\frac{1}{\lambda_{1}^{\frac{6-n}{2}}}
<
\frac{\gamma_{4}}{\gamma_{3}\bar{\lambda}_{1}^{2}},
      \end{split}
    \end{equation*}
while
$\frac{6-n}{2}>0 \; \text{ for } \; n=3,4,5$ and $\bar{\lambda}_{1} \gg 1$ is large.
\end{enumerate}
\end{enumerate}
Thus in any case
$$k>0 \; \text{ on } \; B_{\delta}^{-,+}
\; \text{ and } \; k<0
\; \text{ on } \;
A_{D_{1},D_{2}}^{-,+}.$$
We now restrict to $a_{1}=\bar{a}_{1},v=0$.
Then  due to \eqref{Bubble_Definition} and \eqref{bound} we may assume
$$
0<c<u=\alpha+\alpha_1\varphi_{\bar{a}_{1},\lambda_{1}}
\in \bar{U}_{\bar{a}_{1}}(1,\varepsilon)
$$
and in view of \eqref{faithful_expansion_of_k_in_explicit_form_1},
fixing $0<-r=\tau \ll 1$ sufficiently small, that
$$
u \in B_{D_{2}}^{-,+}
\; \text{ for all } \;
l_{1}\leq\lambda_{1} <\infty
$$
for some $l_{1}>0$, while
$$
u\in
\left\{
\begin{matrix*}[l]
A_{D_{1},D_{2}}^{-,+} &  \;\text{for}\;  &  \lambda_{1}=l_{1}  \\
B_{\delta}^{-,+} &  \;\text{for}\;  & \lambda_{1} \gg l_{1}.
\end{matrix*}
\right.
$$
By continuity  we deduce
$$
k_{u_{1}}=0 \; \text{ for } \; u_{1}=\alpha + \alpha_{1}\varphi_{\bar{a}_{1},\lambda_{1}^{\prime}}
\; \text{ and some } \; \lambda_{1}^{\prime}>l_{1},
$$
while of course
$-r_{u_{1}}=\tau>0,\; 0<u_1\in C^{\infty}$ and $\Vert u_{1} \Vert_{L^{\frac{2n}{n-2}}}=1$. Thus
$$
u_1\in A_{\delta,D_{1}}^{-,+}
\cap E,
\;
E= \{ k=0 \} \cap \{ r<0 \} \cap \{ \Vert \cdot \Vert_{L^{\frac{2n}{n-2}}}=1 \} \subset C^{\infty}(M,\R_{>0}).
$$
In the same way,
considering $\bar{U}_{\bar{a}_{2}}(1,\varepsilon)$ instead of $\bar{U}_{\bar{a}_{1}}(1,\varepsilon)$,
see \eqref{set_u},
we find
$$u_{2}\in \bar{U}_{\bar{a}_{2}}(1,\varepsilon)\cap E.$$
Clearly the connected components of $E$, generated by $u_{1}$ and $u_{2}$ are distinct, since every path within $E$, connecting
$u_{1}$ to $u_{2}$, would have to pass through either of the corresponding
$(D_{1},D_{2})$-annuli, upon which $k<0$.
\end{proof}

\begin{remark} A few comments are in order.
\begin{enumerate}[label=(\roman*)]
\item $K_{dp}$ as a suitable double-peaked function induces at least
two connected components of $E$.
Likewise a suitable $m$-peaked function $K_{mp}$ will give rise to at
least $m$-many connected components of $E$.
\item Such $m$-peaked functions are of type
$$
K_{mp}=\bar{\alpha} + \sum_{j=1}^{m}\bar{\varphi}_{a_{j},\lambda_{j}}
$$
with suitable peak functions $\bar{\varphi}$, and each connected
components is found on functions of type
$$
u=\alpha + \alpha_{1}\varphi_{a_{1},\lambda_{1}} \; \text{ with } \;
a_{1} \; \text{ close to a peak } \; \bar{a}_{j}
$$
with a bubbling function $\varphi$. But to make the argument work, the
constant function $\bar{\alpha}$ has to be close to a certain value, see \eqref{bar_alpha}.
We conjecture, that for a $m$-peaked function $K_{mp}$ as above we will
also find different connected components of $E$, but on functions of
type
$$
u=\alpha + \sum_{i=1}^{q}\alpha_{i}\varphi_{a_{i},\lambda_{i}} \;
\text{ with each } \; a_{i} \; \text{ close to a distinct peak } \;
\bar{a}_{j},
$$
provided $q<p$ and $\bar{\alpha}=\bar{\alpha}(q)$ is chosen
appropriately.
\item Of course the double peak function $K=K_{dp}$ cannot satisfy an A-B-inequality, since $E\neq\emptyset$, cf. Corollary \ref{cor_exit_set_versus_A_B_inequality}, and
$$\frac{\sup_M K}{\inf_{M\setminus\Omega}(-K)}\geq \frac{\sup_M K}{\sup_{M}(-K)} = (4n(n-1))^{\frac{n}{n-2}}(\frac{c_1}{|M|})^{\frac{2}{n-2}}$$
for any $\Omega\supset\supset\{K\geq0\}$, cf. Proposition \ref{prop_A_B_inequality_from_A_B_conditions}. Hence the existence result (iii) of Theorem \ref{thm_exit_set_example},
proved via topological obstruction, seems out of reach of smallness assumption based arguments as in
\cite{Aubin_Bismuth,Mayer_Zhu_Negative_Yamabe_1,Rauzy_Existence}.
\end{enumerate}
\end{remark}

\section{Appendix}
Here we prove Proposition \ref{prop_exp}.

\begin{proof}[\textbf{Proof of (\ref{norm})}]
A simple expansion shows
\begin{equation}\label{expansion_of_norm_up_to_o(v^2)}
\begin{split}
1
= \; &
\|u\|_{L^{\frac{2n}{n-2}}}^{\frac{2n}{n-2}}=\int (\alpha+\alpha_1\varphi_{1}+v)^{\frac{2n}{n-2}}d\mu_{g_0}
\\
= \; &
\int (\alpha+\alpha_1\varphi_{1})^{\frac{2n}{n-2}}d\mu_{g_0} + \frac{2n}{n-2}\int (\alpha+\alpha_1\varphi_{1})^{\frac{n+2}{n-2}}vd\mu_{g_0}
\\
& +
\frac{n(n+2)}{(n-2)^2}\int (\alpha+\alpha_1\varphi_{1})^{\frac{4}{n-2}}v^2d\mu_{g_0}+o_{\|v\|}(\|v\|^2).
\end{split}
\end{equation}
For the principal term in \eqref{expansion_of_norm_up_to_o(v^2)} we obtain
\begin{equation*}
\begin{split}
\int (\alpha  + \alpha_1\varphi_{1})^{\frac{2n}{n-2}}d\mu_{g_0}
= \; &
|M|\alpha^{\frac{2n}{n-2}} + c_1\alpha_1^{\frac{2n}{n-2}} + \frac{2n}{n-2}\int\alpha^{\frac{n+2}{n-2}}(\alpha_1\varphi_{1})d\mu_{g_0}\\
&
+ \frac{2n}{n-2}\int\alpha(\alpha_1\varphi_{1})^{\frac{n+2}{n-2}}d\mu_{g_0} +
\int  \mathcal{R}_{0}d\mu_{g_0}+O(\lambda_1^{-n})\\
=\;& |M|\alpha^{\frac{2n}{n-2}}+c_1\alpha_1^{\frac{2n}{n-2}} + \frac{2n}{n-2}\alpha^{\frac{n+2}{n-2}}\alpha_1\int\varphi_{1}d\mu_{g_0}\\
&
+ \frac{2n}{n-2}b_0\frac{\alpha\alpha_1^{\frac{n+2}{n-2}}}{\lambda_1^{\frac{n-2}{2}}}
+ o_{\frac{1}{\lambda_1}}(\lambda_1^{\frac{2-n}{2}}+\lambda_1^{-2}),
\end{split}
\end{equation*}
where $b_0=\underset{\R^{n}}{\int}\frac{dx}{(1+r^{2})^{\frac{n+2}{2}}}$, cf. Lemma \ref{lem_interactions}, and
\begin{equation*}
  \mathcal{R}_{0}=(\alpha + \alpha_1\varphi_{1})^{\frac{2n}{n-2}}-(\alpha^{\frac{2n}{n-2}}+(\alpha_1\varphi_{1})^{\frac{2n}{n-2}}
 + \frac{2n}{n-2}(\alpha^{\frac{n+2}{n-2}}(\alpha_1\varphi_{1})+\alpha(\alpha_1\varphi_{1})^{\frac{n+2}{n-2}}))
\end{equation*}
with
\begin{equation*}
|\int\mathcal{R}_{0}d\mu_{g_0}|
=
o_{\frac1{\lambda_1}}(\lambda_1^{\frac{2-n}{2}}+\lambda_1^{-2}).
\end{equation*}
By Lemmata \ref{lem_L_g_0_of_bubble} and \ref{lem_optimal_choice1},
\begin{equation}\label{v_linear}
  \int  v d\mu_{g_0}=0,\; \; \int \varphi_{1}^{\frac{n+2}{n-2}}vd\mu_{g_0} =    o_{\frac{1}{\lambda_1}}(\lambda_1^{\frac{2-n}{2}} + {\lambda_1^{-2}}+\|v\|^2),
\end{equation}
whence for the $v$-linear terms in \eqref{expansion_of_norm_up_to_o(v^2)} we find
\begin{equation*}
\begin{split}
\int (\alpha+\alpha_1\varphi_{1})^{\frac{n+2}{n-2}}v d\mu_{g_0}
= \; &
\int (\alpha^{\frac{n+2}{n-2}} + \alpha_1^{\frac{n+2}{n-2}}\varphi_{1}^{\frac{n+2}{n-2}} + \mathcal{R}_1)v
d\mu_{g_0}\\
= \; &
\int \mathcal{R}_1vd\mu_{g_0} + o_{\frac{1}{\lambda_1}}(\lambda_1^{\frac{2-n}{2}} + {\lambda_1^{-2}}+\|v\|^2),
 \end{split}
\end{equation*}
where
\begin{equation}\label{def_R_1}
\begin{split}
\mathcal{R}_1
=
O(\alpha^{\frac{4}{n-2}}\inf(\alpha, \alpha_1\varphi_{1})
 +
(\alpha_1\varphi_{1})^{\frac{4}{n-2}}\inf(\alpha_1\varphi_{1}, \alpha)
).
\end{split}
\end{equation}
Using
\begin{flalign}
&
(1) \quad
\int \inf(\alpha, \alpha_1\varphi_{1})vd\mu_{g_0}
=
O(\frac{\|v\|}{\lambda_1^{\frac{n+2}{4}}})
=
o_{\frac{1}{\lambda_1}}(\lambda_1^{\frac{2-n}{2}}+\lambda_1^{-2}+\|v\|^2)
& \label{R_1_first}
\end{flalign}	
\vspace{-10pt}
\begin{flalign}
&
(2) \quad
\int \varphi_{1}^{\frac{4}{n-2}}\inf(\alpha, \alpha_1\varphi_{1})vd\mu_{g_0}
=O(\frac{\|v\|}{\lambda_1^{\frac{n+2}{4}}}) + O(\frac{\|v\|}{\lambda_1^{\frac{n-2}{2}}}) \label{R_1_second} \\ &\quad\quad\quad\quad\quad\quad\quad\quad\quad\quad\quad\quad\quad~~\;  = o_{\frac{1}{\lambda_1}}(\lambda_1^{\frac{2-n}{2}}+\lambda_1^{-2}+\|v\|^2),
& \notag
\end{flalign}	
we find
\begin{eqnarray*}
\int (\alpha+\alpha_1\varphi_{1})^{\frac{n+2}{n-2}}vd\mu_{g_0}
=
o_{\frac{1}{\lambda_1}}(\lambda_1^{\frac{2-n}{2}}+\lambda_1^{-2}+\|v\|^2).
\end{eqnarray*}
Similarly for the $v$-quadratic term in \eqref{expansion_of_norm_up_to_o(v^2)} we have
\begin{equation*}
\begin{split}
\int (\alpha+& \alpha_1\varphi_{1})^{\frac{4}{n-2}} v^2d\mu_{g_0}
\\
= \; &
\alpha^{\frac{4}{n-2}}\int v^2d \mu_{g_0} + \alpha_1^{\frac{4}{n-2}}\int v^2\varphi_{1}^{\frac{4}{n-2}}d\mu_{g_0} + o_{\frac1{\lambda_1}}(\|v\|^2).
\end{split}
\end{equation*}
Collecting terms we conclude
\begin{equation*}
\begin{split}
 1 =
 \|u\|_{L^{\frac{2n}{n-2}}}^{\frac{2n}{n-2}}
 = \; &
 |M|\alpha^{\frac{2n}{n-2}} + c_1\alpha_1^{\frac{2n}{n-2}} +
\frac{2n\alpha^{\frac{n+2}{n-2}}\alpha_1}{n-2}\int\varphi_{1}d\mu_{g_0}
+
\frac{2n}{n-2}b_0
\frac{\alpha\alpha_1^{\frac{n+2}{n-2}}}{\lambda_1^{\frac{n-2}{2}}} \\
& +
\frac{n(n+2)}{(n-2)^2}(\alpha^{\frac{4}{n-2}}\int v^2d\mu_{g_0}
+
\alpha_1^{\frac{4}{n-2}}\int v^2\varphi_{1}^{\frac{4}{n-2}}d\mu_{g_0}) \\
&  +
o_{\frac{1}{\lambda_1}}(\lambda_1^{\frac{2-n}{2}}+\lambda_1^{-2}+\|v\|^2).
\qedhere
\end{split}
\end{equation*}
\end{proof}

\begin{proof}[\textbf{Proof of (\ref{exp_r}})]
Applying Lemmata \ref{lem_L_g_0_of_bubble} and \ref{lem_optimal_choice1}, we have
\begin{equation*}
\begin{split}
r
= \; &
\int L_{g_0}(\alpha + \alpha_1\varphi_{1}+v) (\alpha+\alpha_1\varphi_{1} + v)d\mu_{g_0}
\\
= \; &
-|M|\alpha^2 + 4n(n-1)c_1\alpha_1^2 - 2\alpha\alpha_1\int \varphi_{1}d\mu_{g_0} \\
& + \int L_{g_0}vvd\mu_{g_0} + o_{\frac{1}{\lambda_1}}(\lambda_1^{-2})
.
\qedhere
\end{split}
\end{equation*}
\end{proof}

\begin{proof}[\textbf{Proof of (\ref{exp_k})}]
Recalling
$$u=\alpha+\alpha_1\varphi_{1}+v\; \text{ with }\; d(\bar{a}_1, a_1)<\varepsilon
\; \text{ and } \;
\Vert u \Vert_{L^{\frac{2n}{n-2}}}=1,$$
a simple expansion shows
\begin{align}\label{expansion_of_k_up_to_o_v^2}
k
= \; &
\int Ku^{\frac{2n}{n-2}}d\mu_{g_0} = \int (-\bar{\alpha} + \sum_{i=1}^2\bar{\varphi}_i)(\alpha + \alpha_1\varphi_{1} + v)^{\frac{2n}{n-2}}
d\mu_{g_0}\nonumber\\
= \; &
-\bar{\alpha}+\sum_{i=1}^2\int \bar{\varphi}_i(\alpha + \alpha_1\varphi_{1})^{\frac{2n}{n-2}}
d\mu_{g_0} +
\frac{2n}{n-2}\sum_{i=1}^2\int \bar{\varphi}_i(\alpha + \alpha_1\varphi_{1})^{\frac{n+2}{n-2}}v
d\mu_{g_0}\nonumber\\
& +
\frac{n(n+2)}{(n-2)^2}\sum_{i=1}^2\int \bar{\varphi}_i(\alpha + \alpha_1\varphi_{1})^{\frac{4}{n-2}}v^2d\mu_{g_0}+o_{\|v\|}(\|v\|^2)\\
=\; &
-\bar{\alpha} + I_1 + I_2 + I_3 + o_{\|v\|}(\|v\|^2).\nonumber
\end{align}
As in the proof of \eqref{norm} we have
\begin{equation*}
\begin{split}
I_1
= \; &
\alpha^{\frac{2n}{n-2}}\sum_{i=1}^2\int \bar{\varphi}_id\mu_{g_0}
+
\alpha_1^{\frac{2n}{n-2}}\int \bar{\varphi}_1\varphi_{1}^{\frac{2n}{n-2}}d\mu_{g_0}\\
& + \frac{2n}{n-2}(\alpha^{\frac{n+2}{n-2}}\alpha_1\int\bar{\varphi}_1\varphi_{1}d\mu_{g_0}
+ \alpha\alpha_1^{\frac{n+2}{n-2}}\int\bar{\varphi}_1\varphi_{1}^{\frac{n+2}{n-2}}d\mu_{g_0})
\end{split}
\end{equation*}
up to some
\begin{equation}\label{error}
  o_{\bar{\lambda}_1d(a_1, \bar{a}_1) + \sfrac{1}{\lambda_1}+\sfrac{1}{\bar{\lambda}_1}}(\bar{\lambda}_1^2d^2(a_1, \bar{a}_1)
+
\lambda_1^{\frac{2-n}{2}} +\frac{\bar{\lambda}_1^2}{\lambda_1^2} ).
\end{equation}
Moreover, up to the same error, by expansion
we find
\begin{equation*}
\begin{split}
\int \bar{\varphi}_1\varphi_{1}^{\frac{2n}{n-2}}d\mu_{g_0}
= \; &
\int_{B_{\epsilon}(a_1)} (\bar{\varphi}_1(a_1)+\nabla\bar{\varphi}_1(a_1)\cdot (x-a_1)
\\
& \hspace{12pt}
+\frac12\nabla^2\bar{\varphi}_1(a_1)(x-a_1, x-a_1))(\frac{\lambda_1}{1+\lambda_1^{2}\gamma_{n} G_{a_1}^{\frac{2}{2-n}}})^{n}d\mu_{g_{a_1}}\\
=\;&
  \bar{\varphi}_1(a_1)c_1 + \frac {c'_2 }{2n}\frac{\Delta\bar{\varphi}_1(a_1)}{\lambda_1^2}
  =  c_1-c_1\bar{\lambda}_1^2d^2(a_1, \bar{a}_1)
-c_4\frac{\bar{\lambda}_1^2}{\lambda_1^2},
\end{split}
\end{equation*}
see also \eqref{estimates_derivatives_of_bar_phi}, and
\begin{equation*}
  \int \bar{\varphi}_1\varphi_{1}^{\frac{n+2}{n-2}}d\mu_{g_0}
=
\int_{B_{\epsilon}(a_1)} \bar{\varphi}_1(a_1)(\frac{\lambda_1}{1+\lambda_1^{2}\gamma_{n} G_{a_1}^{\frac{2}{2-n}}})^{\frac{n+2}2}d\mu_{g_{a_1}}
= b_0\lambda_1^{\frac{2-n}{2}}
\end{equation*}
where $b_0$ and $c_1$ are defined in Lemma \ref{lem_interactions}  and
$c_4=\int_{\mathbb{R}^n}\frac{|x|^2}{(1+|x|^2)^n}dx$.
Moreover
\begin{equation*}
\begin{split}
\int \bar{\varphi}_{1}\varphi_{1}d\mu_{g_0}
= \; &
\int_{B_{\epsilon}(a_1)}
\frac{1}{1+\bar{\lambda}_{1}^{2}|x-\bar{a}_1|^{2}}
(\frac{\lambda_{1}}{1+\lambda_{1}^{2}|x-a|^{2}})^{\frac{n-2}{2}} dx\\
= \; &
\lambda_{1}^{-\frac{n+2}{2}}
\int_{B_{\epsilon \lambda_{1}}(0)}\frac{1}{1+\vert \frac{\bar{\lambda}_{1}}{\lambda_{1}}x+\bar{\lambda}_1(a_1-\bar{a}_1) \vert^{2} }
(\frac{1}{1+|x|^{2}})^{\frac{n-2}{2}} dx
\\
=\; &
\lambda_{1}^{-\frac{n+2}{2}}
\int_
{B_{\epsilon \lambda_{1}}(0)
\setminus
B_{\frac{\lambda_{1}}{\bar{\lambda}_{1}}}(0)}
\frac{1}{1+\vert \frac{\bar{\lambda}_{1}}{\lambda_{1}}x+\bar{\lambda}_1(a_1-\bar{a}_1) \vert^{2} }
(\frac{1}{1+r^{2}})^{\frac{n-2}{2}}dx\\
&
+
\lambda_{1}^{-\frac{n+2}{2}}
\int_
{B_{\frac{\lambda_{1}}{\bar{\lambda}_{1}}}(0)}
\frac{1}{1+\vert \frac{\bar{\lambda}_{1}}{\lambda_{1}}x+\bar{\lambda}_1(a_1-\bar{a}_1) \vert^{2} }
(\frac{1}{1+r^{2}})^{\frac{n-2}{2}}dx \\
=\; &
\Phi_{1} + \Phi_{2}
\end{split}
\end{equation*}
up to some $o_{\sfrac{1}{\lambda}_1+\sfrac{1}{\bar{\lambda}_1}}(\lambda_1^{\frac{2-n}{2}})$,
where
\begin{equation*}
\begin{split}
\Phi_{1}
\lesssim \; &
\lambda_{1}^{-\frac{n+2}{2}}
\vert \frac{\lambda_{1}}{\bar{\lambda}_{1}} \vert^{2}
\int^{ \epsilon \lambda_{1}}
_{\frac{\lambda_{1}}{\bar{\lambda}_{1}}}
r^{n-1-2-(n-2)}dr\\
&\;
=
\frac{1}{\bar{\lambda}_{1}^{2}\lambda_{1}^{\frac{n-2}{2}}}
\ln r|^{r= \epsilon \lambda_{1}}_{r=\frac{\lambda_{1}}{\bar{\lambda}_{1}}}
=
\frac{\ln(\epsilon \lambda_{1})-\ln(\frac{\lambda_{1}}{\bar{\lambda}_{1}})}{\bar{\lambda}_{1}^{2}\lambda_{1}^{\frac{n-2}{2}}}
=
o_{\sfrac{1}{\bar{\lambda}_{1}}}(\frac{1}{\lambda_{1}^{\frac{n-2}{2}}})
\end{split}
\end{equation*}
and
\begin{equation*}
\begin{split}
\Phi_{2}
\lesssim \; &
\lambda_{1}^{-\frac{n+2}{2}}
\int^{\frac{\lambda_{1}}{\bar{\lambda}_{1}}}_{0}
\frac{r^{n-1}}{1+r^{n-2}}dr\\
=\; &
\lambda_{1}^{-\frac{n+2}{2}}
\int^{\frac{\lambda_{1}}{\bar{\lambda}_{1}}}_{1}
\frac{r^{n-1}}{r^{n-2}}dr
+
O(\lambda_{1}^{-\frac{n+2}{2}})
=
o_{\sfrac{1}{\lambda_{1}}+\sfrac{1}{\bar{\lambda}_{1}}}(\frac{1}{\lambda_{1}^{\frac{n-2}{2}}}).
\end{split}
\end{equation*}
Finally
$$
\int \bar{\varphi}_id\mu_{g_0}=O^+(\bar{\lambda}_i^{-2})
$$
and collecting terms we conclude, that
\begin{equation}\label{I_1_estimate_for_k}
\begin{split}
I_1
= \; &
O^+(\bar{\lambda}_i^{-2})
+
c_1\alpha_1^{\frac{2n}{n-2}} \\
& -
c_1\alpha_1^{\frac{2n}{n-2}}\bar{\lambda}_1^2d^2(a_1, \bar{a}_1)
-
c_4\alpha_1^{\frac{2n}{n-2}}\frac{\bar{\lambda}_1^2}{\lambda_1^2}
+
\frac{2n}{n-2}b_0\frac{\alpha\alpha_1^{\frac{n+2}{n-2}}}{\lambda_1^{\frac{n-2}{2}}}
\end{split}	
\end{equation}
up to some error as in \eqref{error}.
Concerning the $v$-linear term in \eqref{expansion_of_k_up_to_o_v^2}, we have
\begin{equation*}
\begin{split}
  I_2 = \; & \sum_{i=1}^{2} \int \bar{\varphi}_i(\alpha^{\frac{n+2}{n-2}} + \alpha_1^{\frac{n+2}{n-2}}\varphi_{1}^{\frac{n+2}{n-2}} + \mathcal{R}_1)v d\mu_{g_0}\\
 = \; & \alpha^{\frac{n+2}{n-2}}\sum_{i=1}^{2}\int \bar{\varphi}_ivd\mu_{g_0} + \alpha_1^{\frac{n+2}{n-2}}\int \bar{\varphi}_1\varphi_{1}^{\frac{n+2}{n-2}}vd\mu_{g_0} + \int \bar{\varphi}_1\mathcal{R}_1vd\mu_{g_0},
\end{split}
\end{equation*}
cf. \eqref{def_R_1}. Note, that
\begin{eqnarray*}
\int \bar{\varphi}_i vd\mu_{g_0}
=
O(\frac{\|v\|}{\bar{\lambda}_i^{2}})=o_{\sfrac{1}{\bar{\lambda}_i}}(\bar{\lambda}_i^{-2}+\|v\|^2),
\end{eqnarray*}
while by direct calculation
\begin{equation}\label{estimates_derivatives_of_bar_phi}
\begin{split}
|\int ( & \bar{\varphi}_1  -\bar{\varphi}(a_1))\varphi_{1}^{\frac{n+2}{n-2}} v d\mu_{g_0}|
\leq
\| |\bar{\varphi}_1-\bar{\varphi}(a_1)|\varphi_{1}^{\frac{n+2}{n-2}} \|_{L^{\frac{2n}{n+2}}}\|v\|_{L^{\frac{2n}{n-2}}}\\
\lesssim\; & |\nabla\bar{\varphi}_1(a_1)|(\int_{B_{\epsilon}(0)}|x|^{\frac{2n}{n+2}}(\frac{\lambda_1}{1+\lambda_1^2|x|^2})^ndx)^{\frac{n+2}{2n}}\|v\|\\
&
+
\sup_{B_{\epsilon}(0)} |\nabla^2\bar{\varphi}_1|(\int_{B_{\epsilon}(0)}|x|^{\frac{4n}{n+2}}(\frac{\lambda_1}{1+\lambda_1^2|x|^2})^ndx)^{\frac{n+2}{2n}}\|v\|
+
O(\frac{\Vert v \Vert}{\lambda_{1}^{\frac{n+2}{2}}})
\\
= \;
&  O(\frac{\bar{\lambda}_1}{\lambda_1}\cdot\bar{\lambda}_1d(a_1, \bar{a}_1)\cdot\|v\|)+O(\frac{\bar{\lambda}^2_1}{\lambda^2_1}\cdot\|v\|)
+
O(\frac{\Vert v \Vert}{\lambda_{1}^{\frac{n+2}{2}}})
 \\
 = \; &
o_{\frac{\bar{\lambda}_1}{\lambda_1}+\bar{\lambda}_1d(a_1, \bar{a}_1)}
(\bar{\lambda}_1^2d^2(a_1, \bar{a}_1) + \lambda_{1}^{\frac{2-n}{2}}  + \frac{\bar{\lambda}_1^2}{\lambda_1^2}  + \|v\|^2)
 \end{split}
\end{equation}
and thus from \eqref{v_linear}
\begin{equation*}
\begin{split}
\int  \bar{\varphi}_1  \varphi_{1}^{\frac{n+2}{n-2}}vd\mu_{g_0}
= \; &
\bar{\varphi}_1(a_1)\int\varphi_{1}^{\frac{n+2}{n-2}}vd\mu_{g_0} + \int (\bar{\varphi}_1-\bar{\varphi}(a_1))\varphi_{1}^{\frac{n+2}{n-2}} v d\mu_{g_0}\\
= \; &
o_{\frac{\bar{\lambda}_1}{\lambda_1}+\bar{\lambda}_1d(a_1, \bar{a}_1)}
(\bar{\lambda}_1^2d^2(a_1, \bar{a}_1) + \lambda_1^{\frac{2-n}{2}} + \frac{\bar{\lambda}_1^2}{\lambda_1^2}  + \|v\|^2).
 \end{split}
\end{equation*}
Moreover, since $\bar{\varphi}_1$ is bounded, we may apply \eqref{R_1_first} and \eqref{R_1_second} to estimate
\begin{equation*}
  \int \bar{\varphi}_1\mathcal{R}_1vd\mu_{g_0}=o_{\frac{1}{\lambda_1}}(\lambda_1^{\frac{2-n}{2}}+\lambda_1^{-2}+\|v\|^2).
\end{equation*}
Collecting terms, we conclude
\begin{equation}\label{I_2_estimate_for_k}
I_2
=
o_{ \frac{\bar{\lambda}_1}{\lambda_1} + \bar{\lambda}_1d(a_1, \bar{a}_1)+\sum_{i}\sfrac{1}{\bar{\lambda}_i} }(\bar{\lambda}_i^{-2} + \bar{\lambda}_1^2d^2(a_1, \bar{a}_1)
+ \lambda_1^{\frac{2-n}{2}} + \frac{\bar{\lambda}_1^2}{\lambda_1^2}  + \|v\|^2).
\end{equation}
Finally up to some $o_{\frac{\bar{\lambda}_1}{\lambda_1}+\bar{\lambda}_1d(a_1, \bar{a}_1)+\sum_{i}\frac{1}{\bar{\lambda}_i}}(\|v\|^2)$
there holds
\begin{equation}\label{I_3_estimate_for_k}
\begin{split}
 I_3 =\; & \sum_{i=1}^2\int \bar{\varphi}_i(\alpha^{\frac{4}{n-2}} + \alpha_1^{\frac{4}{n-2}}\varphi_{1}^{\frac{4}{n-2}}) v^2 d\mu_{g_0}\\
 =\; & \alpha_1^{\frac{4}{n-2}}\bar{\varphi}_1(a_1)\int \varphi_{1}^{\frac{4}{n-2}}v^2d\mu_{g_0}
 = \alpha_1^{\frac{4}{n-2}}\int \varphi_{1}^{\frac{4}{n-2}}v^2d\mu_{g_0}.
 \end{split}
\end{equation}
Recalling \eqref{expansion_of_k_up_to_o_v^2}, from \eqref{I_1_estimate_for_k}, \eqref{I_2_estimate_for_k} and \eqref{I_3_estimate_for_k} we conclude
\begin{equation*}
\begin{split}
k
= \; &
-
\bar{\alpha}
+
O^+(\bar{\lambda}_i^{-2})
+
c_1\alpha_1^{\frac{2n}{n-2}}
-
c_1\alpha_1^{\frac{2n}{n-2}}\bar{\lambda}_1^2d^2(\bar{a}_1, a_1)
-
c_4\alpha_1^{\frac{2n}{n-2}} \frac{\bar{\lambda}_1^2}{\lambda_1^2}
\\
& +
\frac{2n}{n-2}b_0\frac{ \alpha\alpha_1^{\frac{n+2}{n-2}}}{\lambda_1^{\frac{n-2}{2}}}
+
\frac{n(n+2)}{(n-2)^2}\alpha_1^{\frac{4}{n-2}}\int v^2\varphi_{1}^{\frac{4}{n-2}}d\mu_{g_0}\\
& +
o_{ \frac{\bar{\lambda}_1}{\lambda_1} + \bar{\lambda}_1d(a_1, \bar{a}_1)+\|v\|+\sum_{i}\sfrac{1}{\bar{\lambda}_i}}
(\bar{\lambda}_1^2d^2(a_1, \bar{a}_1)+ \lambda_1^{\frac{2-n}{2}} + \frac{\bar{\lambda}_1^2}{\lambda_1^2}  +\|v\|^2).
\qedhere
\end{split}
\end{equation*}
\end{proof}

\extrafootertext{
{
The authors have no conflict of interest to declare.
Data sharing is not applicable to this article as no datasets were generated or analysed during the current study.}
}


\begin{thebibliography}{99}

\bibitem{Aubin_Sur_Le_Problem}
{
Aubin T.
\emph{Sur le problem de la courbure scalaire prescrite.}
Bull. Sci. Math, 1994, 118. p. 465-474
}



\bibitem{Aubin_Bismuth}
{
Aubin T., Bismuth S.
\emph{Courbure scalaire prescrite sur les varietes riemanniennes compactes dans le cas negatif.}
J. Funct. Anal. 143 (1997), No.2, 529-541
}




\bibitem{Bahri_Critical_Points_At_Infininty_In_The_Variational_Calculus}
{
Bahri A.
\emph{
Critical points at infinity in the variational calculus.}
Seminaire equations aux derivates partielles (Polytechnique), 1985-86, No.21,
p.1-31
}

\bibitem{Conley_Book}
{
Conley C.
\emph{Isolated Invariant Sets and the Morse Index.}
Regional Conference Series in Mathematics, Vol. 38, 1978,
Am.Math.Soc.
}



\bibitem{Guenther_Conformal_Normal_Coordinates}
{
G\"unther M.
\emph{Conformal normal coordinates.}
Ann. Global Anal. Geom. 11 (1993), No.2, 173-184
}





\bibitem{Kazdan_Warner_JDE}
{
Kazdan J., Warner F.
\emph{Scalar Curvature and conformal deformation of
Riemannian structure.}
J. Differ. Geom. 10 (1975) 113-134
}



\bibitem{Lee_Parker_Yamabe_Problem}
{
Lee J., Parker T.
\emph{The Yamabe problem.}
Bull. Amer. Math. Soc. (N.S.) 17 (1987), No.1, 37-91
}

\bibitem{MM1}
{
Malchiodi A., Mayer M.
\emph{Prescribing Morse Scalar Curvatures: Blow-Up Analysis.}
Int. Mat. Res. Not., Vol. 2021 Issue 16, p 12532-12612
}



\bibitem{MM3}
{
Mayer M.
\emph{Prescribing Morse scalar curvatures: Critical points at infinity.}
Advances in Calculus of Variations, vol. 15, No.2, 2022, pp. 151-190
}

\bibitem{Mayer_Scalar_Curvature_Flow}
{
Mayer M.
\emph{A scalar curvature flow in low dimensions.}
Calc. Var. Partial Differential Equations 56 (2017), No.2, Art. 24, 41 pp
}




\bibitem{Mayer_Zhu_Negative_Yamabe_1}
{Mayer M., Zhu C.
\emph{Prescribing scalar curvatures: on the negative Yamabe case.}  	
https://doi.org/10.48550/arXiv.2302.02435}

\bibitem{Mayer_Zhu_Negative_Yamabe_3}
{Mayer M., Zhu C.
\emph{Prescribing Morse scalar curvatures:
non uniqueness.}
To appear}


\bibitem{Ouyang}
{
Ouyang T.
\emph{On the positive solutions of semilinear equations $\Delta u+\lambda u^p=0$
on compact manifolds.}
Part II. Indiana Univ. Math. J. 40 (1991), 1083-1141
}

\bibitem{Rauzy_Existence}
{
Rauzy A.
\emph{Courbures scalaires des varietes d'invariant conforme negatif.}
Trans. Amer. Math. Soc. 347 (1995), No.12, 4729-4745
}

\bibitem{Rauzy_Multiplicity}
{
Rauzy A.
\emph{Multiplicite pour un probleme de courbure scalaire prescrite.}
Bull. Sc. Math. 120 (1996) 153-194
}


\bibitem{Rey}
{Rey O.
\emph{ The role of the Green's function in a nonlinear elliptic equation involving the critical Sobolev exponent.} 
J. Funct. Anal. 89 (1990), No.1.}

\bibitem{Vazquez_Veron}
{
Vazquez J.L., Vdron L.
\emph{Solutions positives d'\'equations elliptiques semilin\'eaires sur des vari\'et\'es Riemanniennes compactes.}
C.R. Acad. Sci. Paris 312 (1991) 811-815
}


\end{thebibliography}
\end{document}